\documentclass[12pt]{amsart}

\usepackage{amsfonts, amsthm, amsmath}
\usepackage{rotating}

\usepackage{tikz}

\usepackage{graphics}

\usepackage{amssymb}

\usepackage{amscd}

\usepackage[latin2]{inputenc}

\usepackage{t1enc}

\usepackage[mathscr]{eucal}

\usepackage{indentfirst}

\usepackage{graphicx}

\usepackage{graphics}

\usepackage{pict2e}

\usepackage{mathrsfs}

\usepackage{enumerate}
\usepackage[pagebackref]{hyperref}
\usepackage{cite}
\usepackage{color}
\usepackage{epic}
\usepackage{hyperref} %this gives clickable references, which is nice
\usepackage{framed}
\usepackage{mathabx}
\usepackage{booktabs}
\usepackage{makecell}
\newcolumntype{V}{!{\vrule width 2pt}}

\numberwithin{equation}{section}
\def\blue{\textcolor{blue}}
\def\red{\textcolor{red}}
\def\magenta{\textcolor{magenta}}
\theoremstyle{plain}

\newtheorem{theorem}{Theorem}[section]

\newtheorem{lemma}[theorem]{Lemma}

\newtheorem{proposition}[theorem]{Proposition}

\theoremstyle{definition}

\newtheorem{Def}[theorem]{Definition}

\newtheorem{example}[theorem]{Example}

\newtheorem{conj}[theorem]{Conjecture}

\newtheorem{remark}[theorem]{Remark}

\newtheorem{?}[theorem]{Problem}

\def\des{\mathsf{des}}
\def\S{\mathfrak{S}}
\def\iar{\mathsf{iar}}
\def\lir{\mathsf{lir}}
\def\comp{\mathsf{comp}}
\def\DES{\mathrm{DES}}

\def\LMAX{\mathrm{LMAX}}
\def\LMIN{\mathrm{LMIN}}

\def\LMIN{\mathrm{LMIN}}
\def\st{\mathsf{st}}

\newcommand{\id}{\mathrm{id}}
\def\DESB{\mathrm{DESB}}

\def\rc{\mathrm{rc}}

\def\id{\mathrm{id}}
\def\DESB{\mathrm{DESB}}
\newcommand{\DT}{\mathfrak{DT}}
\def\iop{\mathsf{iop}}
\def\top{\mathsf{top}}
\def\riop{\mathsf{riop}}
\def\rtop{\mathsf{rtop}}
\def\pop{\mathsf{pop}}
\def\lop{\mathsf{lop}}
\def\rpop{\mathsf{rpop}}
\def\rlop{\mathsf{rlop}}
\def\iom{\mathsf{iom}}
\def\tom{\mathsf{tom}}
\def\riom{\mathsf{riom}}
\def\rtom{\mathsf{rtom}}
\def\pom{\mathsf{pom}}
\def\lom{\mathsf{lom}}
\def\rpom{\mathsf{rpom}}
\def\rlom{\mathsf{rlom}}

\def\omi{\mathsf{omi}}

\def\idr{\mathsf{idr}}

\begin{document}

\title[A combinatorial bijection on di-sk trees]{A combinatorial bijection on di-sk trees}

\author[S. Fu]{Shishuo Fu}
\address[Shishuo Fu]{College of Mathematics and Statistics, Chongqing University, Huxi campus, Chongqing 401331, P.R. China}
\email{fsshuo@cqu.edu.cn}

\author[Z. Lin]{Zhicong Lin}
\address[Zhicong Lin]{Research Center for Mathematics and Interdisciplinary Sciences, Shandong University, Qingdao 266237, P.R. China}
\email{linz@sdu.edu.cn}

\author[Y. Wang]{Yaling Wang}
\address[Yaling Wang]{School of Mathematical Sciences, Dalian University of Technology, Dalian 116024, P.R. China}
\email{wyl032021@163.com}

\date{\today}

\begin{abstract}
A di-sk tree is a rooted binary tree whose nodes are labeled by $\oplus$ or $\ominus$, and no node has the same label as its right child. The di-sk trees are in natural bijection with separable permutations. We construct a combinatorial bijection on di-sk trees proving  the two quintuples $(\LMAX,\LMIN,\DESB,\iar,\comp)$ and $(\LMAX,\LMIN,\DESB,\comp,\iar)$ have the same distribution over separable permutations. Here for a permutation $\pi$, $\LMAX(\pi)/\LMIN(\pi)$ is the set of values of the left-to-right maxima/minima of $\pi$ and $\DESB(\pi)$ is the set of descent bottoms of $\pi$, while $\comp(\pi)$ and $\iar(\pi)$ are respectively  the number of components of $\pi$ and the length of initial ascending run of $\pi$. 

Interestingly, our bijection specializes to a bijection on $312$-avoiding permutations, which provides  (up to the classical {\em Knuth--Richards bijection}) an alternative approach to a result of Rubey (2016) that asserts the  two triples $(\LMAX,\iar,\comp)$ and $(\LMAX,\comp,\iar)$ are equidistributed  on $321$-avoiding permutations. Rubey's result is a symmetric extension of an equidistribution due to Adin--Bagno--Roichman, which implies the class of $321$-avoiding permutations with a prescribed number of components is Schur positive. 

Some equidistribution results for various statistics concerning tree traversal are presented in the end.

\end{abstract}

%\subjclass[2010]{05A05, 05A15, 05A19, 15B36}


\maketitle

%\tableofcontents

%%%%%%%%%%%%%%%%%%%%%%%%%%%%%%%%%%%%%
\section{Introduction}\label{sec1: intro}
%%%%%%%%%%%%%%%%%%%%%%%%%%%%%%%%%%%%%

The {\em (large) Schr\"oder numbers } 
$$
S_n=\sum_{i=0}^n\frac{1}{i+1}\binom{2i}{i}\binom{n+i}{n-i}
$$
are one of the most fundamental integer sequences in mathematics. They arise in many classical combinatorial enumeration problems~\cite{kit,shapi2,stan}. In pattern avoidance, one of  the most important classes of permutations, known as {\em separable permutations}, are counted by the Schr\"oder numbers~\cite{shapi,wes,sk}. 

There are two distinct ways to define the separable permutations. One is in terms of pattern avoiding permutations. Let $\S_n$ be the set of permutations of $[n]:=\{1,2,\ldots,n\}$. A permutation $\pi\in\S_n$ is said to {\em avoid} a permutation (or a pattern) $\sigma\in\S_k$, $k\leq n$, if there exists no subsequence of $\pi$ that is order isomorphic to $\sigma$. Separable permutations are permutations that avoid both the patterns $2413$ and $3142$. 

Another description of separable permutations is via two elementary  operations, called direct sum and skew sum of permutations. 
The \emph{direct sum} $\pi\oplus\sigma$ and the \emph{skew sum} $\pi\ominus\sigma$, of $\pi\in\S_k$ and $\sigma\in\S_l$, are permutations in $\S_{k+l}$ defined respectively as
$$
(\pi\oplus\sigma)_i=
\begin{cases}
\pi_i, &\text{for $1\leq i\leq k$};\\
\sigma_{i-k}+k, &\text{for $k+1\leq i\leq k+l$}.
\end{cases}
$$
and
$$
(\pi\ominus\sigma)_i=
\begin{cases}
\pi_i+l, &\text{for $1\leq i\leq k$};\\
\sigma_{i-k}, &\text{for $k+1\leq i\leq k+l$}.
\end{cases}
$$
 For instance, we have $123\oplus 21=12354$ and $123\ominus 21=34521$. The following  characterization of separable permutations is  folkloric  (see~\cite[p.~57--58]{kit}) in pattern avoidance. 
 
 \begin{proposition}\label{desides}
 A permutation is separable if and only if it can be built from the permutation $1$ by applying the operations $\oplus$ and $\ominus$ repeatedly. 
 \end{proposition}

This characterization induces a natural bijection~\cite{shapi} between separable permutations and di-sk trees, where a {\em di-sk tree} is a rooted binary tree whose nodes are labeled by $\oplus$ or $\ominus$, and no node has the same label as its right child. The trees considered in this paper will all be di-sk trees. The main objective of this paper is to construct a combinatorial bijection on di-sk trees that proves a symmetric quintuple equidistribution on separable permutations involving two Comtet statistics \cite{flw}, the number of components and the length of the initial ascending run.  

For a permutation $\pi=\pi_1\pi_2\cdots\pi_n\in\S_n$, define six statistics
\begin{align*}
\LMAX(\pi) &:=\{\pi_i:\pi_j<\pi_i,\: \forall 1\le j < i\};\\
\LMIN(\pi) &:=\{\pi_i:\pi_j>\pi_i,\: \forall 1\le j < i\};\\
\DES(\pi)&:=\{i\in[n-1]:\pi_i<\pi_{i+1}\};\\
\DESB(\pi)&:=\{\pi_{i+1}: i\in\DES(\pi)\};\\
\iar(\pi)&:=\min(\DES(\pi)\cup\{n\});\\
\comp(\pi)&:=|\{i: \forall j\leq i,\: \pi_j\leq i\}|;
\end{align*}
called {\em the set of values of  left-to-right maxima}, {\em the set of values of  left-to-right minima}, {\em the set of positions of descents}, {\em the set of descent bottoms}, {\em the length of initial ascending  run} and {\em the number of components} of $\pi$, respectively.  For  a (finite) collection of patterns $P$, we write $\S_n(P)$ for the set of all permutations in $\S_n$ that avoid simultaneously every pattern contained in $P$. 

\begin{theorem}\label{thm:sep:sym}
There exists an involution $\Phi$  on $\S_n(2413,3142)$ that preserves the triple of set-valued statistics $(\LMAX,\LMIN,\DESB)$ but exchanges the pair $(\comp,\iar)$. Moreover, $\Phi$ restricts to an involution on $\S_n(312)$.
\end{theorem}

The inspiration of Theorem~\ref{thm:sep:sym} stems from the work of Comtet~\cite[Ex.~VI.14]{com} and several recent results. The two statistics $\LMAX/\LMIN$ and $\DESB$ are respectively the set-valued extensions of the classical Stirling and Eulerian statistics, since the number of left-to-right maxima/minima over $\S_n$ gives the {\em Stirling numbers of the first kind} and  the descent polynomial on $\S_n$ is the {\em$n$-th Eulerian polynomial} (see~\cite{bon,FS,pet}). Note that  $\comp(\pi)$ equals the maximum number of components in an expression of $\pi$ as a direct sum of permutations~\cite{adin}. The statistic $\comp$ dates back at least to Comtet~\cite[Ex.~VI.14]{com} and following \cite{flw}, any statistic equidistributed with $\comp$ over a class of restricted permutations will be called a {\em Comtet statistic} over such class. The statistic $\iar$ was considered by Claesson and Kitaev in~\cite{CK}, but under the different notation $\lir$. It was known that 
 \begin{equation}\label{eq:cat}
 |\{\pi\in\S_n(321): \iar(\pi)=k\}|=C_{n,n-k}= |\{\pi\in\S_n(321): \comp(\pi)=k\}|,
 \end{equation}
 where $\{C_{n,k}=\frac{n-k}{n}\binom{n-1+k}{k}\}_{0\le k\le n-1}$ forms the {\em Catalan triangle} (see \cite[A009766]{oeis}). Thus, $\iar$ is a Comtet statistic over $321$-avoiding permutations. 
 
 Recently, Adin, Bagno, and Roichman~\cite{adin} proved a generalization of~\eqref{eq:cat}, which is equivalent to the equidistribution of $(\LMAX,\iar)$ and $(\LMAX,\comp)$ on $321$-avoiding permutations. This result was shown to imply that the class of $321$-avoiding permutations with a prescribed number of components is {\em Schur positive}. Rubey~\cite{rub} later found a symmetric generalization of the Adin--Bagno--Roichman equidistribution via constructing an involution on Dyck paths and using Krattenthaler's bijection~\cite{kra} from Dyck paths to $321$-avoiding permutations. His symmetric equidistribution was shown~\cite{flw} to be equivalent to  the equidistribution of $(\LMAX,\iar,\comp)$ and $(\LMAX,\comp,\iar)$ on $321$-avoiding permutations, up to some elementary transformations on permutations. Since the classical {\em Knuth--Richards bijection} (see~\cite{CK}) between $\S_n(321)$ and $\S_n(312)$ preserves the triple $(\LMAX,\iar,\comp)$ and the fact that $\DESB(\pi)\cup\LMAX(\pi)=[n]$ for $\pi\in\S_n(312)$, the two quadruples 
 $$
 (\LMAX,\DESB,\iar,\comp)\quad \text{and}\quad (\LMAX,\DESB,\comp,\iar)
 $$
 have the same distribution over $\S_n(312)$. 
 
 On the other hand, Claesson, Kitaev, and Steingr\'imsson~\cite[Thm~2.2.48]{kit} constructed a bijection between separable permutations of length $n+1$ with $k+1$ components and {\em Schr\"oder paths} of length $2n$ with $k$ horizontals on the $x$-axis. Combining this bijection with the recent work in~\cite[Thm~3.2]{fw} justifies $\iar$ being a Comtet statistic on separable permutations. 
 
 The above results lead us to find Theorem~\ref{thm:sep:sym}. See also~\cite{flw} for other interesting consequences of Theorem~\ref{thm:sep:sym}.  For other studies of pattern avoiding permutations that emphasize bijective maps, the reader is referred to~\cite{CK,EP,LK,SS}. 

The rest of this paper is mainly devoted to the proof of Theorem~\ref{thm:sep:sym}. In the next section, we transform the involved statistics from separable permutations to di-sk trees, and then in section~\ref{sec:constr of phi}, we construct a combinatorial bijection on di-sk trees to build the involution $\Phi$ for Theorem~\ref{thm:sep:sym}. In section~\ref{sec:tree-tra}, we discuss some further results from the perspective of tree traversal, and derive several new Comtet statistics over di-sk trees. Finally, we conclude our paper by posing several questions for further investigation.

%In this section, we apply a natural bijection of Shapiro and Stephens~\cite{shapi} between di-sk trees and separable permutations 
%
%
%
%Recalled from~\cite{flz} that a rooted  binary tree is called   \emph{di-sk tree} if its nodes are labelled either with  $\oplus$ or $\ominus$ and no node has the same label as its right child (this is called the {\em right chain condition}).

\section{From separable permutations to di-sk trees}
The set of all di-sk trees with $n-1$ nodes is denoted as $\DT_n$. For each $T\in\DT_n$, we use the \emph{inorder} (traversal) to compare nodes on $T$:  starting with the root node, we recursively traverse the left subtree to the parent then to the right subtree if any (see the first tree in Fig.~\ref{8traversal}). We call the first (by inorder) node of $T$ the \emph{inorder root} (abbreviated as {\em iroot} in the sequel) of $T$.

We will apply a natural  bijection $\eta:\S_n(2413,3142)\rightarrow \DT_n$ found by Shapiro and Stephens~\cite{shapi}. The recursive description of $\eta$ recalled below is from~\cite{flz}. Let $\id_1=1$ be the unique permutation of length $1$ and we set $\eta(\id_1)=\emptyset$. For $\pi=\pi_1\ldots\pi_n\in\S_n(2413,3142)$ with $n\geq2$,  find the greatest index $i\in[n-1]$ such that either 
 \begin{equation}\label{eq:greind}
 \min\{\pi_1,\ldots,\pi_i\}>\max\{\pi_{i+1},\ldots,\pi_n\}\quad\text{or}\quad\max\{\pi_1,\ldots,\pi_i\}<\min\{\pi_{i+1},\ldots,\pi_n\}.
 \end{equation}
 In view of Proposition~\ref{desides}, such an index $i$ exists and is unique. We distinguish two cases:
 \begin{itemize}
 \item If the first inequality in \eqref{eq:greind} holds, then $\pi=\omega\ominus\rho$ with $\omega=(\pi_1+i-n)\cdots(\pi_i+i-n)\in\S_i(2413,3142)$ and $\rho=\pi_{i+1}\cdots\pi_{n}\in\S_{n-i}(2413,3142)$. Define $\eta(\pi)=(\eta(\omega),\ominus,\eta(\rho))$, the tree with the left subtree $\eta(\omega)$ and the right subtree $\eta(\rho)$ attached to the root $\ominus$. 
 \item Otherwise,  $\pi=\omega\oplus\rho$, where $\omega=\pi_1\cdots\pi_i\in\S_i(2413,3142)$ and $\rho=(\pi_{i+1}-i)\cdots(\pi_{n}-i)\in\S_{n-i}(2413,3142)$. Then define $\eta(\pi)$ to be the tree with the left subtree $\eta(\omega)$ and the right subtree $\eta(\rho)$ attached to the root $\oplus$.
\end{itemize}
See Fig.~\ref{di-sk} for an example of $\eta$ for $\pi=5\,2\,3\,4\,1\,9\,11\,10\,6\,8\,7\in\S_{11}(2413,3142)$. One of the important features of $\eta$ was proved in~\cite{flz}.

% \begin{figure}
%\begin{tikzpicture}[scale=0.3]
%
%\node at (12,12) {$\circ$};
%\draw[-] (11.85,11.85) to (8,10);
%\draw[-] (12.15,11.85) to (16,10);
%\red{\draw[-] (8,10) to (10.3,7.7);}\blue{\node at (10.7,7.4){$1$};}\node at (8,10) {$\bullet$};
%\draw[-] (8,10) to (6.18,8.18);
%\node at (6,8) {$\circ$};\red{\draw[-] (5.85,7.85) to (3,5);}\blue{\node at (2.6,4.7){$2$};}
%\draw[-] (6.15,7.85) to (8,6);\red{\draw[-] (8,6) to (9,5);}\blue{\node at (9.4,4.7){$3$};}\node at (8,6) {$\bullet$};
%\draw[-] (8,6) to (6.18,4.18);\node at (6,4) {$\circ$};
%\red{\draw[-] (5.85,3.85) to (5,3);} \blue{\node at (4.6,2.7){$4$};}\red{\draw[-] (6.15,3.85) to (7,3);} \blue{\node at (7.4,2.7) {$5$};}
%
%\node at (16,10) {$\bullet$};
%\draw[-] (16,10) to (14.18,8.18);
%\draw[-] (16,10) to (19.82,8.18);
%\node at (14,8) {$\circ$};\red{\draw[-] (13.85,7.85) to (13,7);}\blue{\node at (12.6,6.7){$9$};}
%\draw[-] (14.15,7.85) to (15.82,6.18);\node at (16,6) {$\bullet$};\red{\draw[-] (16,6) to (15,5);}\blue{\node at (14.4,5){$11$};}\red{\draw[-] (16,6) to (17,5);}\blue{\node at (17.6,5){$10$};}
%
%\node at (20,8) {$\circ$};\draw[-] (18.15,7.85) to (20,6);\node at (20,6) {$\bullet$};
%
%\end{tikzpicture}
%\caption{The bijection $\eta:\S_n(2413,4213)\rightarrow \DT_{n}$.\label{di-sk}}
%\end{figure}

 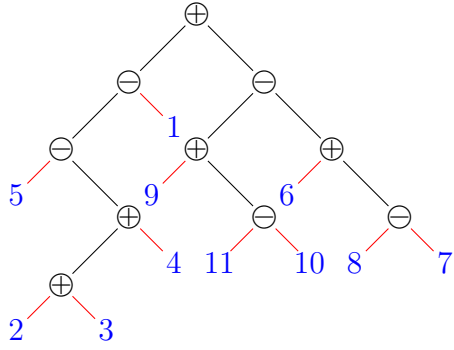
\begin{figure}
\begin{tikzpicture}[scale=0.3]
\draw[-] (17,10) to (19,12);
\draw[-] (20,12) to (22,10);
\draw[-] (22,9) to (20,7);
\draw[-] (23,9) to (25,7);
\draw[-] (20,6) to (22,4);
\draw[-] (26,6) to (28,4);
%\draw[-] (23,3) to (25,1);

\draw[-] (16,9) to (14,7);
\node at (13.5,6.5) {$\ominus$};
\red{\draw[-] (13,6) to (12,5);}
\blue{\node at (11.5,4.5) {$5$};}

\draw[-] (14,6) to (16,4);
\node at (16.5,3.5) {$\oplus$};
\red{\draw[-] (17,3) to (18,2);}
\blue{\node at (18.5,1.5) {$4$};}

\draw[-] (16,3) to (14,1);
\node at (13.5,0.5) {$\oplus$};
\red{\draw[-] (13,0) to (12,-1);}
\blue{\node at (11.5,-1.5) {$2$};}
\red{\draw[-] (14,0) to (15,-1);}
\blue{\node at (15.5,-1.5) {$3$};}

\node at (16.5,9.5) {$\ominus$};
\red{\draw[-] (17,9) to (18,8);}
\blue{\node at (18.5,7.5) {$1$};}

\node at (19.5,12.5) {$\oplus$};
\node at (22.5,9.5) {$\ominus$};
\node at (19.5,6.5) {$\oplus$};
\red{\draw[-] (19,6) to (18,5);}
\blue{\node at (17.5,4.5) {$9$};}

\node at (25.5,6.5) {$\oplus$};
\red{\draw[-] (25,6) to (24,5);}
\blue{\node at (23.5,4.5) {$6$};}

\node at (22.5,3.5) {$\ominus$};
\red{\draw[-] (22,3) to (21,2);}
\blue{\node at (20.5,1.5) {$11$};}
\red{\draw[-] (23,3) to (24,2);}
\blue{\node at (24.5,1.5) {$10$};}

\node at (28.5,3.5) {$\ominus$};
\red{\draw[-] (29,3) to (30,2);}
\blue{\node at (30.5,1.5) {$7$};}
\red{\draw[-] (28,3) to (27,2);}
\blue{\node at (26.5,1.5) {$8$};}
%\node at (20,-3.5) {$2\,4\,5\,3\,1\,9\,11\,10\,6\,8\,7$};
\end{tikzpicture}
\caption{The bijection $\eta:\S_n(2413,3142)\rightarrow \DT_{n}$.\label{di-sk}}
\end{figure}

\begin{lemma}[Theorem~2.3 in~\cite{flz}]
\label{lem:sep1}
The mapping 
 $\eta: \S_n(2413,3142)\rightarrow\DT_{n}$ is a bijection such that 
 \begin{align}\label{minu:des}
 i\in \DES(\pi) \,\,\Longleftrightarrow  \text{ the $i$th node (by inorder) of $\eta(\pi)$ is $\ominus$}
 \end{align}
 for each $\pi\in\S_n(2413,3142)$.
\end{lemma}

For each tree $T\in\DT_n$, let $\iop(T)$ be the number of {\em initial $\oplus$-nodes}  (by inorder) in $T$.
It follows from~\eqref{minu:des}  that 
\begin{equation}\label{eq:iop}
\iar(\pi)-1=\iop(\eta(\pi))
\end{equation}
 for any  $\pi\in\S_n(2413,3142)$. 
 But what is $\comp(\pi)$ corresponding to in the di-sk tree $\eta(\pi)$? Let us consider the {\em spine} of $T$, i.e., the path from the root of $T$ to the iroot of $T$. Let $\top(T)$ be the number of top consecutive $\oplus$-nodes in the spine of $T$. For instance, the spine of $T$ in Fig.~\ref{di-sk} is $\oplus-\ominus-\ominus$ (from the top) and so $\top(T)=1$.
 \red{
 % \begin{remark}
 % Use preorder to define $\top$? A better (top not equal to bottom) example in Fig.~1? Calculate the distribution of postorder (left-right-root).
 % \end{remark}
 }
\begin{lemma}\label{lem:sep2}
For each permutation $\pi\in\S_n(2413,3142)$,  we have  
$$\comp(\pi)-1=\top(\eta(\pi)).
$$
\end{lemma}
\begin{proof}
Recall that $\comp(\pi)-1$ equals the cardinality of the set $\{k\in [n-1]: \forall j\leq k,\: \pi_j\leq k\}$. 
We need to consider two cases:
\begin{itemize}
\item If $\{k\in [n-1]: \forall j\leq k,\: \pi_j\leq k\}=\emptyset$, then the root of $\eta(\pi)$ is a $\ominus$-node and so $\top(\eta(\pi))=0=\comp(\pi)-1$. 
\item Otherwise, let $l$ be the greatest integer in $\{k\in [n-1]: \forall j\leq k,\: \pi_j\leq k\}$. Clearly,  $l$ is the greatest index smaller than $n$ such that~\eqref{eq:greind} holds. Thus, by the construction of $\eta$ we have  $\eta(\pi)=(\eta(\omega),\oplus,\eta(\rho))$ assuming that $\pi=\omega\oplus\rho$ with $\omega=\pi_1\cdots\pi_l$. It then follows by induction on $n$ that 
$$
\comp(\pi)-1=\comp(\omega)=1+\top(\omega)=\top(\pi).
$$
\end{itemize}
In either case, the assertion is true. 
\end{proof}
\begin{remark}
The statistic $\top$ on di-sk trees was previously considered by Corteel, Martinez, Savage and Weselcouch in~\cite[Corollary~5]{cor}, where they constructed a bijection from $021$-avoiding inversion sequences of length $n$ with $k$ initial zeros  to $\{T\in\DT_n: \top(T)=k-1\}$. 
On the other hand, Kim and Lin~\cite{LK} built a bijection from $021$-avoiding inversion sequences to $(2413,4213)$-avoiding permutations which transforms positions of ascents to positions of descents. Combining these two bijections gives
\begin{equation}\label{eq:2413-4213}
|\{T\in\DT_n: \top(T)=k-1\}|=|\{\pi\in\S_n(2413,4213): \iar(\pi)=k\}|. 
\end{equation}
It would be interesting to construct a natural bijection (probably in similar flavor  as $\eta$) between $\DT_n$ and $\S_n(2413,4213)$ that proves~\eqref{eq:2413-4213}. 
See also section~\ref{sec:tree-tra} for another interpretation of $\top$ in terms of tree traversal.
\end{remark}

\section{The construction of \texorpdfstring{$\Phi$}{Phi}}\label{sec:constr of phi}

This section is devoted to the construction of $\Phi$. We begin with an elementary operation on di-sk tree that will be used frequently during our construction of $\Phi$. 

Let $T$ be a di-sk tree and $v$ be an $\ominus$-node of $T$. We introduce the di-sk tree $\mathcal{L}(v,T)$ whenever  there exists an $\oplus$-node $w$ ($w$ is not a right child) in the following two situations: 
\begin{itemize}
\item $v$ is the left child of $w$. 
\item $v$ is the right child of an $\oplus$-node, denoted $w'$, whose parent is $w$. 
\end{itemize}
In either case, define $\mathcal{L}(v,T)$ to be the di-sk tree obtained from $T$ by cutting the $\oplus$-node $w$, together with its right subtree (if any), and inserting it as the left child of $v$. The original left child of $v$ (if any) becomes the left child of $w$, while the original right parent of $w$ (if any) becomes the right parent of $v$ (resp.~$w'$) for the first case (resp.~the second case), keeping the remaining nodes and edges of $T$ unchanged. See Fig.~\ref{left-opera} for the illustration of $\mathcal{L}(v,T)$ in the above two cases. Since the edges we have inserted/deleted in the construction of $\mathcal{L}(v,T)$ are all left edges, we see that $\mathcal{L}(v,T)$ is still a di-sk tree. Moreover, both cases are seen to be invertible and if $T':=\mathcal{L}(v,T)$, we will denote the inverse map as $\mathcal{L}^{-1}(v,T')=T$.

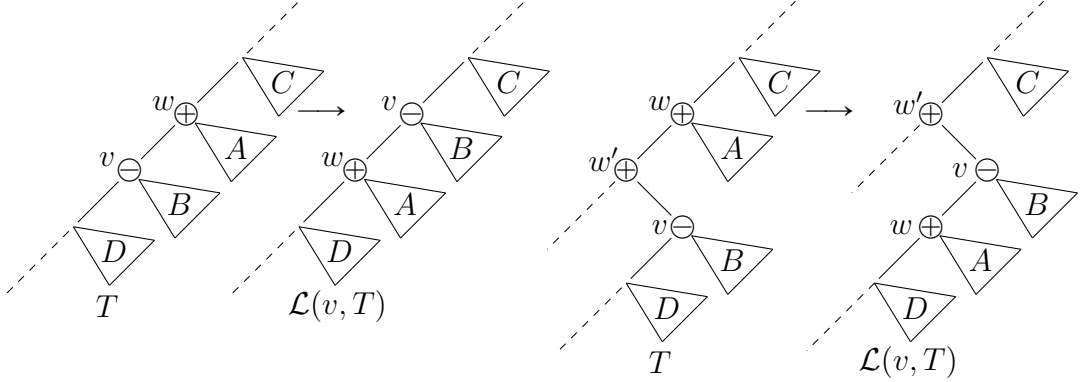
\begin{figure}
\begin{tikzpicture}[scale=0.3]

\draw[-,dashed] (7.7,12.7) to (10,15);\draw[-] (5.5,10.5) to (7.3,12.3);
\draw[-] (7.5,12.5) to (9.1,9.9);\draw[-] (7.5,12.5)  to (11.1,11.9);\draw[-] (9.1,9.9) to (11.1,11.9);
\node at (9.3,11.4) {$C$};
\node at (5,10) {$\oplus$};\node at (4,10.5) {$w$};\draw[-] (4.5,9.5) to (3,8);
\draw[-] (5.4,9.6) to (7,7);\draw[-] (5.4,9.6) to (9,9);\draw[-] (7,7) to (9,9);
\node at (7.2,8.5) {$A$};

\node at (2.5,7.5) {$\ominus$};\node at (1.5,8) {$v$};\draw[-,dashed] (-0.2,4.8) to (-3,2);\draw[-] (0.2,5.2) to (2,7);
\draw[-] (0,5) to (1.6,2.4);\draw[-] (0,5)  to (3.6,4.4);\draw[-] (1.6,2.4) to (3.6,4.4);
\node at (1.8,3.9) {$D$};
\draw[-] (2.9,7.1) to (4.5,4.5);\draw[-] (2.9,7.1) to (6.5,6.5);\draw[-] (4.5,4.5) to (6.5,6.5);
\node at (4.7,6) {$B$};

\node at (1.5,1.5) {$T$};\node at (11.7,1.5) {$\mathcal{L}(v,T)$};
%%%%%%%%%%%%%
\node at (11,10) {$\longrightarrow$};
%%%%%%%%%%%%%%%%%%%%%%%%%%%%
\draw[-,dashed] (17.7,12.7) to (20,15);\draw[-] (15.5,10.5) to (17.3,12.3);
\draw[-] (17.5,12.5) to (19.1,9.9);\draw[-] (17.5,12.5)  to (21.1,11.9);\draw[-] (19.1,9.9) to (21.1,11.9);
\node at (19.3,11.4) {$C$};
\node at (15,10) {$\ominus$};\node at (14,10.5) {$v$};\draw[-] (14.5,9.5) to (13,8);
\draw[-] (15.4,9.6) to (17,7);\draw[-] (15.4,9.6) to (19,9);\draw[-] (17,7) to (19,9);
\node at (17.2,8.5) {$B$};

\node at (12.5,7.5) {$\oplus$};\node at (11.5,8) {$w$};\draw[-,dashed] (9.8,4.8) to (7,2);\draw[-] (10.2,5.2) to (12,7);
\draw[-] (10,5) to (11.6,2.4);\draw[-] (10,5)  to (13.6,4.4);\draw[-] (11.6,2.4) to (13.6,4.4);
\node at (11.8,3.9) {$D$};
\draw[-] (12.9,7.1) to (14.5,4.5);\draw[-] (12.9,7.1) to (16.5,6.5);\draw[-] (14.5,4.5) to (16.5,6.5);
\node at (14.7,6) {$A$};
%%%%%%%%%%%%%%%%%%%%%%%%%%%%%%%%%
%%%%%%%%%%%%%%%%%%%%%%%%%%%%%%%%right graph

\draw[-,dashed] (29.8,12.8) to (32,15);\draw[-] (27.5,10.5) to (29.3,12.3);
\draw[-] (29.5,12.5) to (31.1,9.9);\draw[-] (29.5,12.5) to (33.1,11.9);\draw[-] (31.1,9.9)to (33.1,11.9);
\node at (31.3,11.4) {$C$};
\node at (27,10) {$\oplus$};\node at (26,10.5) {$w$};\draw[-] (26.5,9.5) to (25,8);
\draw[-] (27.4,9.6) to (29,7);\draw[-] (27.4,9.6) to (31,9);\draw[-] (29,7) to (31,9);
\node at (29.2,8.5) {$A$};

\node at (24.5,7.5) {$\oplus$};\node at (23.5,8) {$w'$};\draw[-,dashed] (24,7) to (21,4);
\draw[-] (25,7) to (26.5,5.5);\node at (27,5) {$\ominus$};\draw[-,dashed] (24.3,2.3) to (21.5,-0.5);
\node at (26,5) {$v$};

\draw[-] (26.5,4.5) to (24.7,2.7);
\draw[-] (24.5,2.5) to (26.1,-0.1);\draw[-] (24.5,2.5) to (28.1,1.9);\draw[-] (26.1,-0.1) to (28.1,1.9);
\node at (26.3,1.4) {$D$};
\draw[-] (27.4,4.6) to (29,2);\draw[-] (27.4,4.6) to (31,4);\draw[-] (29,2) to (31,4);
\node at (29.2,3.5) {$B$};
\node at (26,-1) {$T$};
%%%%%%%%%%%%%%%%%%%%%%%
\node at (33.5,10) {$\longrightarrow$};
%%%%%%%%%%%%%%%%%%%%%%%

\draw[-,dashed] (40.8,12.8) to (43,15);\draw[-] (38.5,10.5) to (40.3,12.3);
\draw[-] (40.5,12.5) to (42.1,9.9);\draw[-] (40.5,12.5) to (44.1,11.9);\draw[-] (42.1,9.9)to (44.1,11.9);
\node at (42.3,11.4) {$C$};
\node at (38,10) {$\oplus$};\node at (37,10.5) {$w'$};\draw[-,dashed] (37.5,9.5) to (34.5,6.5);
\draw[-] (38.5,9.5) to (40,8);\node at (40.5,7.5) {$\ominus$};\node at (39.3,7.5) {$v$};
\draw[-] (40.9,7.1) to (42.5,4.5);\draw[-] (40.9,7.1) to (44.5,6.5);\draw[-] (42.5,4.5) to (44.5,6.5);
\node at (42.7,6) {$B$};

\draw[-] (37.5,4.5) to (35.7,2.7);
\draw[-] (35.5,2.5) to (37.1,-0.1);\draw[-] (35.5,2.5) to (39.1,1.9);\draw[-] (37.1,-0.1) to (39.1,1.9);
\node at (37.3,1.4) {$D$};

\draw[-] (40,7) to (38.5,5.5);\node at (38,5) {$\oplus$};\draw[-,dashed] (35.5,2.5) to (32.5,-0.5);
\node at (36.7,5) {$w$};
\draw[-] (38.4,4.6) to (40,2);\draw[-] (38.4,4.6)  to (42,4);\draw[-] (40,2) to (42,4);
\node at (40.2,3.5) {$A$};
\node at (37,-1) {$\mathcal{L}(v,T)$};

\end{tikzpicture}
\caption{Two cases to obtain the di-sk tree $\mathcal{L}(v,T)$ from $T$.\label{left-opera}}
\end{figure}

The reason to introduce the transformation $\mathcal{L}$ lies in the following lemma. 
\begin{lemma}\label{lem:L}
Let $\pi=\eta^{-1}(T)$ and $\pi'=\eta^{-1}(\mathcal{L}(v,T))$ for a di-sk tree $T$. 
Then, 
\begin{align*}
\DESB(\pi)&=\DESB(\pi'),\\
\LMAX(\pi)&=\LMAX(\pi')\quad{and}\\
\LMIN(\pi)&=\LMIN(\pi').
\end{align*}
\end{lemma}
\begin{proof}
We need to describe the inverse $\eta^{-1}: \DT_{n}  \rightarrow \S_n(2413,3142)$ of $\eta$. For a given di-sk tree $
T\in\DT_n$, let us add some edges to $T$ so that the out-degree of each node is exactly two (see the red edges in Fig.~\ref{di-sk}). There are $n$ new edges to be added to $T$, thus creating $n$ new leaves. The next step of $\eta^{-1}$ is to assign the integers in $[n]$ to these $n$ new leaves so that for each $\oplus$-node (resp.~$\ominus$-node) of $T$, the integers assigned to the leaves belonging to the left subtree of this node are all smaller (resp.~greater) than those assigned to leaves belonging to the right subtree. Such an assignment is unique and the permutation $\eta^{-1}(T)$ can be derived from reading these $n$ integers by the inorder of this augmented tree (see Fig.~\ref{di-sk}).

From the above description of $\eta^{-1}$, we see that the transformation $\mathcal{L}$ preserves the assignment of the augmented tree, namely, if a new leaf of $T$ has been assigned an integer $k$, then the corresponding new leaf (i.e., the leaf of $\mathcal{L}(v,T)$ added to the corresponding node under $\mathcal{
L}$) receives the same integer $k$. Notice that an integer is a descent bottom of $\pi=\eta^{-1}(T)$ if and only if it is assigned to a leaf appearing immediately after an $\ominus$-node by the inorder of the augmented tree of $T$. Thus, $\DESB(\pi)=\DESB(\pi')$ holds. To see that $\LMAX(\pi)=\LMAX(\pi')$, we divide letters assigned to the augmented tree of $T=\eta(\pi)$ (see Fig.~\ref{left-opera}) into three subclasses:
\begin{enumerate}
	\item[I \;] letters assigned to the right subtree $A$;
	\item[II ] letters assigned to the right subtree $B$;
	\item[III] all the remaining letters.
\end{enumerate}
We observe that: 1) all letters in class I are greater than all letters in class II; 2) none of the letters in class II is a left-to-right maximum. Now $\pi'$ is obtained from $\pi$ by swapping class I letters with class II letters, therefore the status of being a left-to-right maximum or not remains the same for each letter in classes I, II, and III. So $\LMAX(\pi)=\LMAX(\pi')$ as desired. The proof of $\LMIN(\pi)=\LMIN(\pi')$ is similar by noting that none of the letters in class I is a left-to-right minimum.
\end{proof}

%%%%%%%%%%%%%%
Let $\widetilde\DT_n$ be the set of trees $T\in\DT_n$ such that the spine of $T$ has at least one $\ominus$-node. The next result contains the main ingredient for our construction of $\Phi$. 
\begin{theorem}\label{thm:di-sk}
Let $\widetilde\DT_n^{(k,l)}:=\{T\in\widetilde\DT_n: \top(T)=k, \iop(T)=l\}$. If $k\geq1$, then there exists a bijection $\phi:\widetilde\DT_n^{(k,l)}\rightarrow \widetilde\DT_n^{(k-1,l+1)}$ satisfying
\begin{equation}\label{desb:lmax}
(\DESB,\LMAX,\LMIN)\, \eta^{-1}(T)= (\DESB,\LMAX,\LMIN)\, \eta^{-1}(\phi(T))
\end{equation}
for any $T\in\widetilde\DT_n^{(k,l)}$.
\end{theorem}

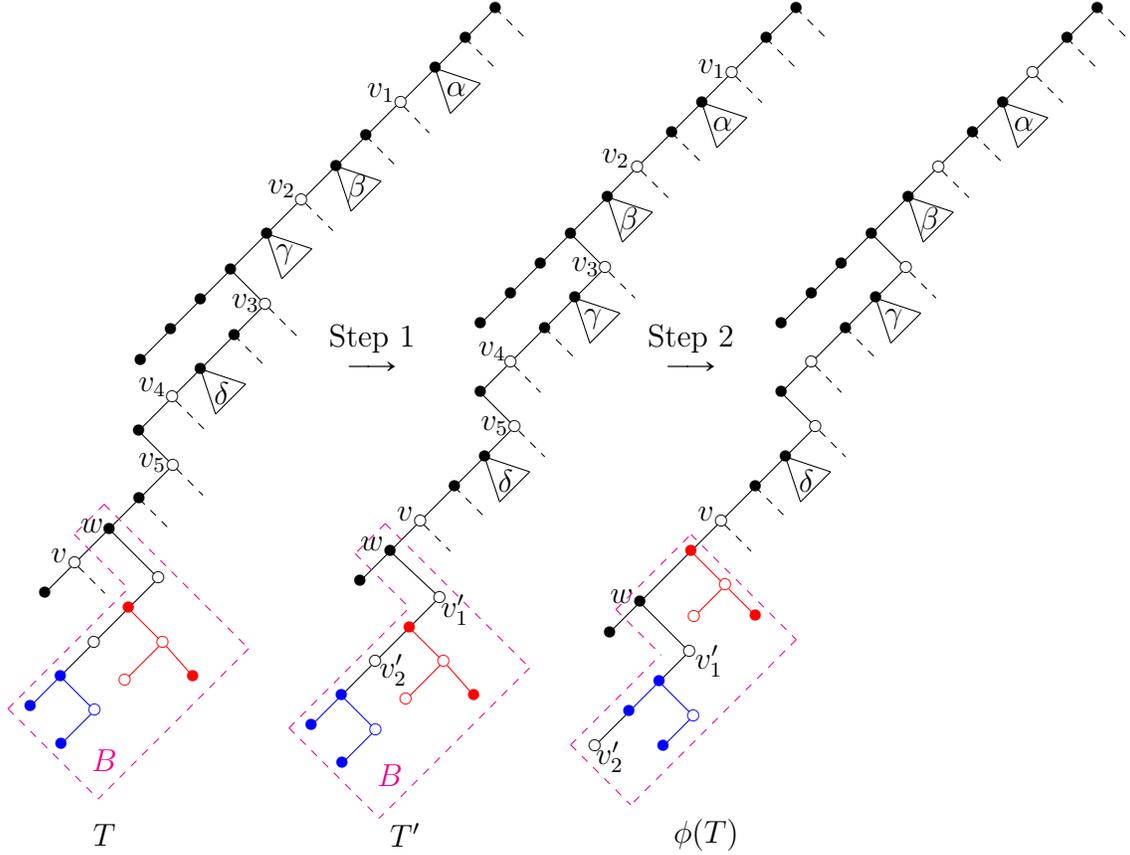
\begin{figure}
\begin{tikzpicture}[scale=0.4]
\node at (10,30) {$\bullet$};\draw[-,dashed] (10,30) to (11,29);
\draw[-] (10,30) to (9,29);
\node at (9,29) {$\bullet$};\draw[-,dashed] (9,29) to (10,28);
\draw[-] (9,29) to (8,28);
\node at (8,28) {$\bullet$};\draw[-] (8,28) to (8.5,26.5);\draw[-] (8.5,26.5) to (9.5,27.5);\draw[-] (8,28) to (9.5,27.5);
\node at (8.7,27.3) {$\alpha$};
\draw[-] (8,28) to (7,27);\node at (6.85,26.85) {$\circ$};\node at (6.2,27.2) {$v_1$};\draw[-,dashed] (6.95,26.74) to (7.8,25.85);
\draw[-] (6.7,26.75) to (5.7,25.75);\node at (5.7,25.75) {$\bullet$};\draw[-,dashed] (5.7,25.75) to (6.7,24.75);
\draw[-] (5.7,25.75) to (4.7,24.75);
\node at (4.7,24.75) {$\bullet$};\draw[-] (4.7,24.75) to (5.2,23.25);\draw[-] (5.2,23.25) to (6.2,24.25);
\draw[-]   (6.2,24.25) to (4.7,24.75);
\node at (5.4,24) {$\beta$};
\draw[-]  (4.7,24.75) to (3.7,23.75);\node at (3.55,23.6) {$\circ$};\node at (2.9,24) {$v_2$};\draw[-,dashed] (3.7,23.5) to (4.6,22.6);
\draw[-]  (3.4,23.5) to (2.4,22.5);\node at (2.4,22.5) {$\bullet$};
\draw[-] (2.4,22.5) to (2.9,21);\draw[-] (2.9,21) to (3.9,22);\draw[-] (2.4,22.5) to (3.9,22);
\node at (3,21.8) {$\gamma$};
\draw[-] (2.4,22.5) to (1.2,21.3);\node at (1.2,21.3) {$\bullet$};
\draw[-] (1.2,21.3) to (0.2,20.3);\node at (0.2,20.3) {$\bullet$};
\draw[-] (0.2,20.3) to (-0.8,19.3);\node at (-0.8,19.3) {$\bullet$};
\draw[-] (-0.8,19.3) to (-1.8,18.3);\node at (-1.8,18.3) {$\bullet$};
%%%第二层
\draw[-] (1.2,21.3) to (2.2,20.3);\node at (2.35,20.15) {$\circ$};\node at (1.7,20.2) {$v_3$};\draw[-,dashed] (2.46,20.04) to (3.45,19.05);
\draw[-] (2.23,20.03) to (1.2,19);\node at (1.32,19.12) {$\bullet$};\draw[-,dashed] (1.32,19.12) to (2.32,18.12);
\draw[-] (1.32,19.12) to (0.2,18);\node at (0.2,18) {$\bullet$};
\draw[-] (0.2,18) to (0.7,16.5);\draw[-] (0.7,16.5) to (1.7,17.5);\draw[-] (0.2,18) to  (1.7,17.5);
\node at (0.9,17.2) {$\delta$};
\node at (5.9,18) {$\longrightarrow$};\node at (5.9,19) {Step 1};
\node at (16.5,18) {$\longrightarrow$};\node at (16.5,19) {Step 2};
\draw[-] (0.2,18) to (-0.6,17.2);\node at (-0.74,17.06) {$\circ$};\node at (-1.4,17.4) {$v_4$};\draw[-,dashed] (-0.62,16.94) to (0.36,15.96);
\draw[-] (-0.86,16.94) to (-1.86,15.94);\node at (-1.86,15.94) {$\bullet$};
\draw[-] (-1.86,15.94) to (-0.86,14.94);\node at (-0.72,14.8) {$\circ$};\node at (-1.4,14.9) {$v_5$};\draw[-,dashed] (-0.6,14.68) to (0.28,13.8);
\draw[-] (-0.84,14.68) to (-1.84,13.68);\node at (-1.84,13.68) {$\bullet$};\draw[-,dashed] (-1.84,13.68) to (-0.84,12.68);
\draw[-]  (-1.84,13.68) to (-2.84,12.68);\node at (-2.84,12.68) {$\bullet$};\node at (-3.4,12.8) {$w$};
\draw[-]   (-2.84,12.68) to (-3.84,11.68);\node at (-3.98,11.54) {$\circ$};\node at (-4.5,11.8) {$v$};\draw[-]  (-4.1,11.42) to (-4.98,10.54);\node at (-4.98,10.54) {$\bullet$};\draw[-,dashed] (-3.86,11.42) to (-2.98,10.54);
%%%第四层
\draw[-]   (-2.84,12.68) to (-1.34,11.18);\node at  (-1.2,11.04) {$\circ$};
\draw[-]   (-1.32,10.92) to (-2.2,10.04);

\draw[-]  (-2.2,10.04) to (-3.2,9.04);\node at  (-3.34,8.9) {$\circ$};\draw[-]   (-3.46,8.78)  to (-4.46,7.78);
\red{
\node at (-2.2,10.04) {$\bullet$};
\draw[-]  (-2.2,10.04) to (-1.2,9.04);\node at  (-1.06,8.9) {$\circ$};\draw[-]  (-0.94,8.78) to (-0.06,7.78);\node at (-0.06,7.78) {$\bullet$};
\draw[-]  (-1.18,8.78) to (-2.18,7.78);\node at  (-2.32,7.64) {$\circ$};
}
\blue{
\node at (-4.46,7.78){$\bullet$};\draw[-]  (-4.46,7.78) to (-3.46,6.78);\node at  (-3.32,6.64) {$\circ$};
\draw[-]  (-3.44,6.52) to (-4.44,5.52);\node at  (-4.44,5.52){$\bullet$};
\draw[-]  (-4.46,7.78) to (-5.46,6.78);\node at  (-5.46,6.78){$\bullet$};
}
%%%方框
\magenta{
\draw[-,dashed] (-3,13.5) to (-4,12.5);
\draw[-,dashed] (-3,13.5) to (1.8,8.7);\draw[-,dashed]  (1.8,8.7) to (-3.2,3.7);
\draw[-,dashed]  (-6.2,6.7) to (-3.2,3.7);\draw[-,dashed]  (-6.2,6.7) to (-2.2,10.7);
\draw[-,dashed]  (-4,12.5) to (-2.2,10.7);
}
\node at  (-3,2.5){$T$};\node at  (7,2.5){$T'$};\node at  (17,2.5){$\phi(T)$};\node at  (6.5,4.5){\magenta{$B$}};\node at  (-3,5){$\magenta{B}$};
%%%%%第二棵树
\node at (20,30) {$\bullet$};\draw[-,dashed] (20,30) to (21,29);
\draw[-] (20,30) to (19,29);
\node at (19,29) {$\bullet$};\draw[-,dashed] (19,29) to (20,28);
\draw[-] (19,29) to (18,28);\node at (17.86,27.86) {$\circ$};\node at (17.2,28) {$v_1$};\draw[-,dashed] (17.98,27.74) to (18.86,26.86);
\draw[-] (17.74,27.74) to (16.86,26.86);
\node at (16.86,26.86){$\bullet$};\draw[-] (16.86,26.86) to (17.36,25.36);\draw[-]  (17.36,25.36) to (18.36,26.36);
\draw[-]  (16.86,26.86) to (18.36,26.36);
\node at (17.56,26.16) {$\alpha$};
\draw[-]  (16.86,26.86) to  (15.86,25.86);\node at (15.86,25.86){$\bullet$};
\draw[-]    (15.86,25.86) to (14.86,24.86);\draw[-,dashed] (15.86,25.86)  to(16.86,24.86);
\node at (14.72,24.72) {$\circ$};\node at (14,25) {$v_2$};\draw[-,dashed] (14.84,24.6)  to(15.74,23.7) ;
\draw[-] (14.6,24.6) to (13.72,23.72);\node at (13.72,23.72){$\bullet$};
\draw[-] (13.72,23.72) to (14.22,22.22);\draw[-]  (14.22,22.22) to  (15.22,23.22);\draw[-] (13.72,23.72) to  (15.22,23.22);
\node at (14.42,22.9) {$\beta$};
\draw[-]  (13.72,23.72) to (12.5,22.5);\node at (12.5,22.5){$\bullet$};
\draw[-]   (12.5,22.5) to (11.5,21.5);\node at (11.5,21.5){$\bullet$};
\draw[-]   (11.5,21.5) to (10.5,20.5);\node at (10.5,20.5){$\bullet$};
\draw[-]   (10.5,20.5) to (9.5,19.5);\node at (9.5,19.5){$\bullet$};
%%第二层
\draw[-] (12.5,22.5) to (13.5,21.5);\node at (13.64,21.36) {$\circ$};\node at (13,21.4) {$v_3$};\draw[-,dashed] (13.78,21.22)  to(14.64,20.36);
\draw[-] (13.52,21.24) to (12.64,20.36);\node at (12.64,20.36){$\bullet$};
\draw[-]  (12.64,20.36) to (13.14,18.86);\draw[-]  (13.14,18.86) to (14.14,19.86);\draw[-]  (12.64,20.36) to (14.14,19.86);
\node at (13.24,19.66) {$\gamma$};
\draw[-]  (12.64,20.36) to (11.64,19.36);\node at (11.64,19.36){$\bullet$};\draw[-,dashed] (11.64,19.36)to(12.64,18.36);
\draw[-]  (11.64,19.36) to  (10.64,18.36);\node at (10.5,18.22) {$\circ$};\node at (9.9,18.6) {$v_4$};\draw[-,dashed] (10.62,18.12)to(11.5,17.22);
\draw[-] (10.38,18.1) to  (9.5,17.22);\node at (9.5,17.22){$\bullet$};
%%第三层
\draw[-]  (9.5,17.22) to (10.5,16.22);\node at (10.64,16.08) {$\circ$};\node at (9.98,16.12) {$v_5$};\draw[-,dashed] (10.78,15.94)to(11.64,15.08);
\draw[-]  (10.52,15.96) to (9.64,15.08);\node at (9.64,15.08){$\bullet$};
\draw[-]  (9.64,15.08) to  (10.14,13.58);\draw[-]   (10.14,13.58) to (11.14,14.58);\draw[-]  (9.64,15.08)to (11.14,14.58);
\node at (10.34,14.28) {$\delta$};
\draw[-]  (9.64,15.08) to  (8.64,14.08);\node at (8.64,14.08){$\bullet$};
\draw[-,dashed]   (8.64,14.08) to (9.64,13.08);
\draw[-]   (8.64,14.08) to (7.64,13.08);\node at (7.5,12.94) {$\circ$};\node at (7,13.2) {$v$};
\draw[-,dashed]   (7.62,12.82) to (8.5,11.94);
%%step 2
\draw[-]  (7.38,12.82) to (6.5,11.94);\node at (6.5,11.94){$\bullet$};\node at (5.9,12.2) {$w$};
\draw[-] (6.5,11.94) to (5.5,10.94);\node at (5.5,10.94){$\bullet$};
%%第四层
\draw[-]   (6.5,11.94) to (8,10.54);\node at  (8.14,10.4) {$\circ$};\node at (8.65,10) {$v_1'$};
\draw[-]   (8.02,10.28) to (7.14,9.4);

\draw[-]  (7.14,9.4) to (6.14,8.4);\node at  (6,8.26) {$\circ$};\node at (6.6,8) {$v_2'$};\draw[-]   (5.88,8.14)  to (4.88,7.14);
\red{
\node at (7.14,9.4) {$\bullet$};
\draw[-]  (7.14,9.4) to (8.14,8.4);\node at  (8.28,8.26) {$\circ$};\draw[-]  (8.4,8.14) to (9.28,7.14);\node at (9.28,7.14) {$\bullet$};
\draw[-]  (8.16,8.14) to (7.16,7.14);\node at  (7.02,7) {$\circ$};
}
\blue{
\node at (4.88,7.14){$\bullet$};\draw[-] (4.88,7.14) to (5.88,6.14);\node at  (6.02,6) {$\circ$};
\draw[-]  (5.9,5.88) to (4.9,4.88);\node at  (4.9,4.88){$\bullet$};
\draw[-]  (4.88,7.14) to (3.88,6.14);\node at  (3.88,6.14){$\bullet$};
}
%%%方框
\magenta{
\draw[-,dashed] (6.34,12.86) to (5.34,11.86);
\draw[-,dashed] (6.34,12.86) to (11.14,8.06);\draw[-,dashed]  (11.14,8.06) to (6.14,3.06);
\draw[-,dashed]  (3.14,6.06) to (6.14,3.06);\draw[-,dashed]  (3.14,6.06) to (7.14,10.06);
\draw[-,dashed]  (5.34,11.86) to (7.14,10.06);
}

%%%%第三棵树
\node at (30,30) {$\bullet$};\draw[-,dashed] (30,30) to (31,29);
\draw[-] (30,30) to (29,29);
\node at (29,29) {$\bullet$};\draw[-,dashed] (29,29) to (30,28);
\draw[-] (29,29) to (28,28);\node at (27.86,27.86) {$\circ$};\draw[-,dashed] (27.98,27.74) to (28.86,26.86);
\draw[-] (27.74,27.74) to (26.86,26.86);
\node at (26.86,26.86){$\bullet$};\draw[-] (26.86,26.86) to (27.36,25.36);\draw[-]  (27.36,25.36) to (28.36,26.36);
\draw[-]  (26.86,26.86) to (28.36,26.36);
\node at (27.56,26.16) {$\alpha$};
\draw[-]  (26.86,26.86) to  (25.86,25.86);\node at (25.86,25.86){$\bullet$};
\draw[-]    (25.86,25.86) to (24.86,24.86);\draw[-,dashed] (25.86,25.86)  to(26.86,24.86);
\node at (24.72,24.72) {$\circ$};\draw[-,dashed] (24.84,24.6)  to(25.74,23.7) ;
\draw[-] (24.6,24.6) to (23.72,23.72);\node at (23.72,23.72){$\bullet$};
\draw[-] (23.72,23.72) to (24.22,22.22);\draw[-]  (24.22,22.22) to  (25.22,23.22);\draw[-] (23.72,23.72) to  (25.22,23.22);
\node at (24.42,22.9) {$\beta$};
\draw[-]  (23.72,23.72) to (22.5,22.5);\node at (22.5,22.5){$\bullet$};
\draw[-]   (22.5,22.5) to (21.5,21.5);\node at (21.5,21.5){$\bullet$};
\draw[-]   (21.5,21.5) to (20.5,20.5);\node at (20.5,20.5){$\bullet$};
\draw[-]   (20.5,20.5) to (19.5,19.5);\node at (19.5,19.5){$\bullet$};
%%第二层
\draw[-] (22.5,22.5) to (23.5,21.5);\node at (23.64,21.36) {$\circ$};\draw[-,dashed] (23.78,21.22)  to(24.64,20.36);
\draw[-] (23.52,21.24) to (22.64,20.36);\node at (22.64,20.36){$\bullet$};
\draw[-]  (22.64,20.36) to (23.14,18.86);\draw[-]  (23.14,18.86) to (24.14,19.86);\draw[-]  (22.64,20.36) to (24.14,19.86);
\node at (23.24,19.66) {$\gamma$};
\draw[-]  (22.64,20.36) to (21.64,19.36);\node at (21.64,19.36){$\bullet$};\draw[-,dashed] (21.64,19.36)to(22.64,18.36);
\draw[-]  (21.64,19.36) to  (20.64,18.36);\node at (20.5,18.22) {$\circ$};\draw[-,dashed] (20.62,18.12)to(21.5,17.22);
\draw[-] (20.38,18.1) to  (19.5,17.22);\node at (19.5,17.22){$\bullet$};
%%第三层
\draw[-]  (19.5,17.22) to (20.5,16.22);\node at (20.64,16.08) {$\circ$};\draw[-,dashed] (20.78,15.94)to(21.64,15.08);
\draw[-]  (20.52,15.96) to (19.64,15.08);\node at (19.64,15.08){$\bullet$};
\draw[-]  (19.64,15.08) to  (20.14,13.58);\draw[-]   (20.14,13.58) to (21.14,14.58);\draw[-]  (19.64,15.08)to (21.14,14.58);
\node at (20.34,14.28) {$\delta$};
\draw[-]  (19.64,15.08) to  (18.64,14.08);\node at (18.64,14.08){$\bullet$};
\draw[-,dashed]   (18.64,14.08) to (19.64,13.08);
\draw[-]   (18.64,14.08) to (17.64,13.08);\node at (17.5,12.94) {$\circ$};\node at (17,13.2) {$v$};
\draw[-,dashed]   (17.62,12.82) to (18.5,11.94);
\draw[-]   (17.38,12.82) to (16.5,11.94);\draw[-]   (16.5,11.94) to (14.8,10.24);\node at (14.2,10.45) {$w$};\node at (14.8,10.24){$\bullet$};
\red{
\node at (16.5,11.94){$\bullet$};\draw[-]   (16.5,11.94) to (17.5,10.94);\node at (17.64,10.8) {$\circ$};
\draw[-]   (17.52,10.68) to (16.74,9.9);\node at (16.6,9.76) {$\circ$};\draw[-]   (17.76,10.68) to (18.64,9.8);\node at (18.64,9.8){$\bullet$};
}
\draw[-] (14.8,10.24) to (13.8,9.24);\node at (13.8,9.24){$\bullet$};
\draw[-] (14.8,10.24) to (16.3,8.74);\node at (16.44,8.6){$\circ$};\node at (17.1,8.3) {$v_1'$};
\draw[-] (16.32,8.48) to (15.44,7.6);
\draw[-] (14.44,6.6) to (13.44,5.6);\node at (13.3,5.46){$\circ$};\node at (13.8,5.1) {$v_2'$};
\blue{
\node at (15.44,7.6){$\bullet$};\draw[-] (15.44,7.6) to (14.44,6.6);\node at (14.44,6.6){$\bullet$};
\draw[-] (15.44,7.6) to (16.44,6.6);\node at (16.58,6.46){$\circ$};
\draw[-] (16.46,6.34) to (15.58,5.46);\node at(15.58,5.46){$\bullet$};
}
%%方框
\magenta{
\draw[-,dashed] (14,10) to (16.5,12.5);\draw[-,dashed]  (16.5,12.5) to  (20,9);\draw[-,dashed]  (20,9) to (14.5,3.5);
\draw[-,dashed]  (12.5,5.5) to (14.5,3.5);\draw[-,dashed]  (12.5,5.5) to (15.5,8.5);\draw[-,dashed] (14,10) to (15.5,8.5);
}

\end{tikzpicture}
\caption{An illustration  of $\phi$, where each $\oplus$-node (resp.~$\ominus$-node) in di-sk trees is replaced by a solid (resp.~hollow) circle, for simplicity.\label{bij:disk}}
\end{figure}

\begin{proof}For a fixed di-sk tree $T\in\widetilde\DT_n^{(k,l)}$ with $k\geq1$,
the construction of $\phi(T)$ can be performed in the following two steps. In the first step, we do the ``{\bf swing down}'' on $T$ (see Step 1 of Fig.~\ref{bij:disk}), i.e.,
\begin{itemize}
\item Find the topmost $\ominus$-node, say $v_1$, on the spine of $T$ (since $T\in\widetilde\DT_n$, such a $v_1$ always exists). Find the first (by inorder) $\ominus$-node, say $v$ (possibly $v=v_1$), of $T$. 
\item In tree $T$, there is a unique path $P$ from $v_1$ to $v$. Let $v_1,v_2,\ldots,v_{k-1},v_k=v$ be all the $\ominus$-nodes on the path $P$ in the order we visit them when walking from $v_1$ to $v$. Note that by our choice of $v$, the path $P$ cannot have two consecutive right edges, making all of the $\ominus$-nodes $v_1,v_2,\ldots,v_{k-1},v_k=v$ on $P$ eligible for applying the transformation $\mathcal{L}$.
\item Define the di-sk tree $T'$ by 
$$
T'=\mathcal{L}(v_k,\mathcal{L}(v_{k-1},\ldots,\mathcal{L}(v_2,\mathcal{L}(v_1,T))\cdots)).
$$
\end{itemize}

For the second step, we do the ``{\bf backward shift}'' on $T'$ to obtain $\phi(T)$ (see the shift inside the dotted box in Step 2 of Fig.~\ref{bij:disk}), i.e.,
\begin{itemize}
\item Let $w$ be the left child of $v$ in $T'$. Then $w$ must be an $\oplus$-node according to the construction of the first step above. 
\item Let $B$ (possibly empty) be the right subtree of $w$. If $B$ is empty, then set $\phi(T)=T'$. Otherwise, the root of $B$ is an $\ominus$-node according to the definition of di-sk trees, as $w$ is an $\oplus$-node. Let $P'$ be the spine of $B$.
\item Let $v_1',v_2',\ldots,v_{\ell}'$ be all the $\ominus$-nodes in the path $P'$ from the top to the bottom. Then $v_1'$ is the topmost node of $P'$, which is also the root of $B$. For $1\leq i\leq \ell-1$, suppose the number of $\oplus$-nodes on $P'$ between $v_i'$ and $v_{i+1}'$ is $c_i$. Suppose the number of $\oplus$-nodes on $P'$ below $v_{\ell}'$ is $c_{\ell}$. For instance, for the middle tree $T'$ in Fig.~\ref{bij:disk}, we have $c_1=1$ and $c_2=2$. 
\item Introduce  $\mathcal{L}^{-k}(v',T'):=\mathcal{L}^{-1}(v',\mathcal{L}^{-k+1}(v',T'))$ recursively for $k\geq2$. Define the di-sk tree $\phi(T)$ by 
$$
\phi(T)=\mathcal{L}^{-c_{\ell}}(v_{\ell}',\mathcal{L}^{-c_{\ell-1}}(v_{\ell-1}',\ldots,\mathcal{L}^{-c_2}(v_2',\mathcal{L}^{-c_1}(v_1',T'))\cdots)).
$$
\end{itemize}

By the above construction, we see that the node immediately after (by inorder) $w$ of $\phi(T)$ is an $\ominus$-node and therefore $\phi(T)$ is a di-sk tree in $\widetilde\DT_n^{(k-1,l+1)}$. We aim to show that $\phi$ is a bijection by defining $\phi^{-1}$ explicitly. 

For a di-sk tree $\tilde T\in\widetilde\DT_n^{(k-1,l+1)}$ with $k\geq1$, $l\ge 0$, we use two steps to obtain $\phi^{-1}(\tilde T)$ from $\tilde T$ as follows. In the first step, we do the ``{\bf forward shift}'' on $\tilde T$ to obtain a di-sk tree $\tilde T^*$, i.e., 
\begin{itemize}
\item Find the $\oplus$-node, say $w$,  immediately before the first $\ominus$-node (by inorder), such a $w$ always exists since $\iop(\tilde T)=l+1\ge 1$. Trace the unique path from $w$ back to the root of $\tilde T$ to locate the first $\ominus$-node, say $v$. Such a $v$ always exists since $\tilde T\in\widetilde\DT_n$. Let $c_1$ be the number of nodes (necessarily $\oplus$-nodes) in the path from $v$ to $w$, with $v$ and $w$ excluded. 
\item Let $\tilde B$ (possibly empty) be the right subtree of $w$ in $\tilde T$. If $\tilde B$ is empty, then $v$ must be the first $\ominus$-node by inorder. In this case, $c_1=0$ and we set $\tilde T^*=\tilde T$. Otherwise, $\tilde B$ is not empty and let us consider the the spine $\tilde P^*$ of $\tilde B$. 
\item Let $v_1',v_2',\ldots,v_{\ell}'$ be all the $\ominus$-nodes in the path $\tilde P^*$ from the top to the bottom. Then $v_1'$ and $v_{\ell}'$ (possibly coincide) are the root and the tail of the path $\tilde P^*$, respectively. For $2\leq i\leq \ell$, let $c_i$ be the number of $\oplus$-nodes between $v_{i-1}'$ and $v_i'$ in the path $\tilde P^*$.
\item Introduce  $\mathcal{L}^{k}(v',\tilde T):=\mathcal{L}(v',\mathcal{L}^{k-1}(v',\tilde T))$ recursively for $k\geq2$. Define the di-sk tree $\tilde T^*$ by 
$$
\tilde T^*=\mathcal{L}^{c_{\ell}}(v_{\ell}',\mathcal{L}^{c_{\ell-1}}(v_{\ell-1}',\ldots,\mathcal{L}^{c_2}(v_2',\mathcal{L}^{c_1}(v_1',\tilde T))\cdots)).
$$
\end{itemize}

We see that $\tilde T^*\in\widetilde\DT_n$ and the ``{\bf forward shift}'' is clearly inverse to the ``{\bf backward shift}''. For the second step,  we do the ``{\bf swing up}'' on $\tilde T^*$ to obatin $\phi^{-1}(\tilde T)$: 
\begin{itemize}
\item From the construction of the ``{\bf forward shift}'',  the $\ominus$-node $v$ becomes  the parent of $w$ in $\tilde T^*$. Let $v_1$ be the topmost $\ominus$-node in the spine of $\tilde T^*$. Such a $v_1$ always exists since $\tilde T^*\in\widetilde\DT_n$.
\item  In the tree $\tilde T^*$, there is a unique path $\tilde P$ from $v_1$ to $v$, which contains no consecutive right edges. Let $v_1,v_2,\ldots,v_{k-1},v_k=v$ be all the $\ominus$-nodes on $\tilde P$ in the order we visit them when walking from $v_1$ to $v$. 
\item Define the di-sk tree $\phi^{-1}(\tilde T)$ by 
$$
\phi^{-1}(\tilde T)=\mathcal{L}^{-1}(v_1,\mathcal{L}^{-1}(v_{2},\ldots,\mathcal{L}^{-1}(v_{k-1},\mathcal{L}^{-1}(v_k,\tilde T^*))\cdots)).
$$
\end{itemize}

Since the ``{\bf swing up}'' is inverse to the ``{\bf swing down}'', the mapping $\phi^{-1}$ is  indeed the inverse of $\phi$. As every step of the bijection $\phi$ involves only the elementary transformation $\mathcal{L}$, Lemma~\ref{lem:L} guarantees that the desired property~\eqref{desb:lmax} holds. This ends the proof of the theorem. 
\end{proof}

\begin{example}
As an example of $\phi$,  the di-sk tree $T\in\widetilde\DT_{11}^{(1,0)}$ in Fig.~\ref{di-sk} and its image $\phi(T)\in\widetilde\DT_{11}^{(0,1)}$ under $\phi$ are drawn in Fig.~\ref{exam:phi}. One can check that $\eta^{-1}(T)=\pi=5\,2\,3\,4\,1\,9\,11\,10\,6\,8\,7$, $\eta^{-1}(\phi(T))=\pi'=5\,9\,6\,8\,7\,11\,10\,2\,3\,4\,1$,
$$
(\DESB,\LMAX,\LMIN)\, \pi=(\{1,2,6,7,10\},\{5,9,11\},\{5,2,1\})=(\DESB,\LMAX,\LMIN)\, \pi',
$$
 $(\comp,\iar)\,\pi=(2,1)$ and $(\comp,\iar)\,\pi'=(1,2)$.
\begin{figure}
\begin{tikzpicture}[scale=0.2]
\draw[-] (17,8) to (19,10);
\draw[-] (20,10) to (22,8);
\draw[-] (22,7) to (20,5);
\draw[-] (23,7) to (25,5);
\draw[-] (20,4) to (22,2);
\draw[-] (26,4) to (28,2);
\draw[-] (16,7) to (14,5);
\node at (13.5,4.5) {$\ominus$};
\draw[-] (14,4) to (16,2);
\node at (16.5,1.5) {$\oplus$};
\draw[-] (16,1) to (14,-1);
\node at (13.5,-1.5) {$\oplus$};
\node at (16.5,7.5) {$\ominus$};
\node at (19.5,10.5) {$\oplus$};
\node at (22.5,7.5) {$\ominus$};
\node at (19.5,4.5) {$\oplus$};
\node at (25.5,4.5) {$\oplus$};
\node at (22.5,1.5) {$\ominus$};
\node at (28.5,1.5) {$\ominus$};

\node at (31,6.5) {$\longrightarrow$};\node at (31,8) {$\phi$};
%right graph
\node at (44.5,10.5) {$\ominus$};
\draw[-] (44,10) to (42,8);\node at (41.5,7.5) {$\ominus$};
\draw[-] (44,4) to (42.5,2.5);\node at (42,2) {$\oplus$};
\draw[-] (42,7) to (44,5);\node at (44.5,4.5) {$\oplus$};
\draw[-] (41,7) to (38,4);\node at (37.5,3.5) {$\oplus$};
\draw[-] (37,3) to (35,1);\node at (34.5,0.5) {$\oplus$};
\draw[-] (38,3) to (40,1);\node at (40.5,0.5) {$\ominus$};
\draw[-] (35,0) to (37,-2);\node at (37.5,-2.5) {$\ominus$};
\draw[-] (38,-3) to (40,-5);\node at (40.5,-5.5) {$\oplus$};
\draw[-] (41,-6) to (43,-8);\node at (43.5,-8.5) {$\ominus$};
\end{tikzpicture}
\caption{An example of the bijection $\phi$.\label{exam:phi}}
\end{figure}
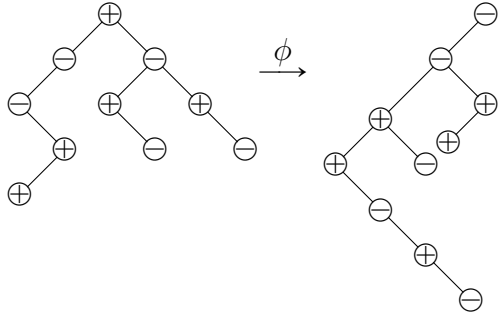
\end{example}

Now we are ready to prove our main result, Theorem~\ref{thm:sep:sym}. 

\begin{proof}[{\bf Proof of Theorem~\ref{thm:sep:sym}}] The involution $\Phi$ is defined recursively using $\phi$. Set $\Phi(\id_1)=\id_1$. For each $\pi=\pi_1\pi_2\cdots\pi_n\in\S_n(2413,3142)$ and $n\geq2$, we need to distinguish two cases as follows. 
\begin{itemize}
\item If $\eta(\pi)\in\DT_n\setminus \widetilde \DT_n$, i.e., the spine of $\eta(\pi)$ is composed of $\oplus$-nodes only, hence $\pi=1\oplus\sigma$, where $\sigma=(\pi_2-1)(\pi_3-1)\cdots(\pi_n-1)\in\S_{n-1}(2413,3142)$. Define $\Phi(\pi)=1\oplus\Phi(\sigma)$.
\item Otherwise, we have $\eta(\pi)\in\widetilde \DT_n^{(k,l)}$ for some $k$ and $l$. Let us define 
$$\Phi(\pi)=\eta^{-1}\circ\phi^{k-l}\circ\eta(\pi).$$
\end{itemize}

For instance, if $\pi=2\,4\,5\,9\,6\,8\,7\,11\,10\,3\,1\in\S_{11}(2413,3142)$, then $\eta(\pi)\in\widetilde \DT_{11}^{(0,3)}$ and so
$$
\Phi(\pi)=\eta^{-1}\circ\phi^{-3}\circ\eta(\pi)=2\,1\,4\,3\,5\,9\,11\,10\,6\,8\,7.
$$
It follows from Lemmas~\ref{lem:sep1},~\ref{lem:sep2} and Theorem~\ref{thm:di-sk} that $\Phi$ is an involution that preserves the triple of set-valued statistics $(\LMAX,\LMIN,\DESB)$ but exchanges the pair $(\comp,\iar)$, which completes the proof of the first statement of Theorem~\ref{thm:sep:sym}. 

For the second statement, we observe that $\pi\in\S_n(312)$ if and only if each $\ominus$-node of $\eta(\pi)$ has no right child. It is clear that the elementary transformation $\mathcal{L}$, which deletes/inserts left edges only, preserves this kind of property. Consequently, $\Phi$ indeed restricts to an involution on $\S_n(312)$.
\end{proof}

\begin{remark}
Since $\Phi$ restricts to an involution on ${\S_{n}(312)}$, we get immediately that the two quintuples 
$(\LMAX,\LMIN,\DESB,\comp,\iar)$ and $(\LMAX,\LMIN,\DESB,\iar,\comp)$ have the same distribution over $\S_n(312)$. However, as $\LMIN(\pi)=\{1,2,\ldots,\pi_1\}$ for any $\pi\in\S_n(312)$ and  $\pi_1=\min(\LMAX(\pi))$,  this result is equivalent to Theorem 1.1~(ii) of~\cite{flw}.
\end{remark}

\section{tree traversal}\label{sec:tree-tra}

In retrospect, we note that for each di-sk tree $T$, the statistic $\top(T)$ also equals the number of initial $\oplus$-nodes in $T$ when we use the {\em preorder} (i.e., recursively traversing the parent to the left subtree then to the right subtree) instead of the inorder to traverse $T$. Motivated by this new perspective, we investigate in this  section the distributions of the number of initial $\oplus$-nodes with respect to the following eight types of tree traversal (see Tab.~\ref{eight-tree-tra}). The first six of which are usually called {\em depth first traversals}, while the last two are called {\em breadth first traversals}.

\begin{table}[bp]\caption{Eight types of tree traversal for a di-sk tree $T$ and their associated statistics for the number of initial $\oplus$-nodes}\label{eight-tree-tra}
\centering
\begin{tabular}{|c||c|c|c|c|}
\hline
Name & inorder & right inorder & preorder & right preorder \\
\hline
Rule & Left--Root--Right & Right--Root--Left & Root--Left--Right & Root--Right--Left \\
\hline
Stat & $\iop(T)$ & $\riop(T)$ & $\top(T)$ & $\rtop(T)$ \\
\hline
Name & postorder & right postorder & level order & right level order \\
\hline
Rule & Left--Right--Root & Right--Left--Root & Left--Right--Next level & Right--Left--Next level \\
\hline
Stat & $\pop(T)$ & $\rpop(T)$ & $\lop(T)$ & $\rlop(T)$ \\
\hline
\end{tabular}
\end{table}

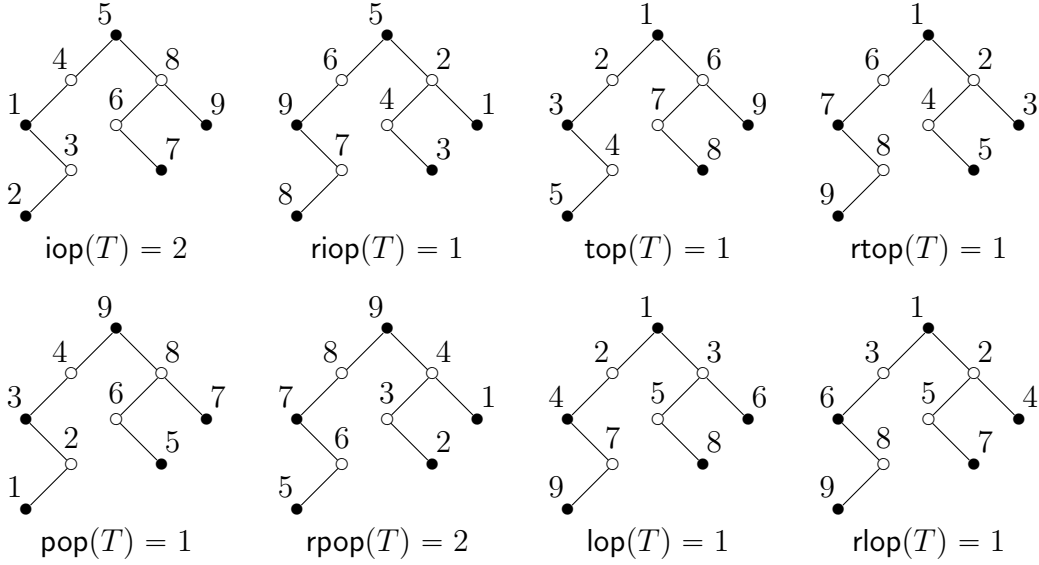
\begin{figure}
\begin{tikzpicture}[scale=0.3]
\node at (2.5,19.5) {$\bullet$};
\node at (4.5,21.5) {$\circ$};
\node at (6.5,23.5) {$\bullet$};
\node at (2.5,15.5) {$\bullet$};
\node at (4.5,17.5) {$\circ$};
\node at (6.5,19.5) {$\circ$};
\node at (8.5,21.5) {$\circ$};
\node at (8.5,17.5) {$\bullet$};
\node at (10.5,19.5) {$\bullet$};
\draw[-] (2.7,19.7) to (4.33,21.33);
\draw[-] (4.7,21.7) to (6.33,23.33);
\draw[-] (2.7,15.7) to (4.33,17.33);
\draw[-] (6.7,19.7) to (8.33,21.33);
\draw[-] (2.67,19.33) to (4.3,17.7);
\draw[-] (6.67,23.33) to (8.3,21.7);
\draw[-] (6.67,19.33) to (8.3,17.7);
\draw[-] (8.67,21.33) to (10.3,19.7);
%\red{\draw[-] (1.5,18.5) to (2.33,19.33);}
%\red{\draw[-] (1.5,14.5) to (2.33,15.33);}
%\red{\draw[-] (5.5,18.5) to (6.33,19.33);}
%\red{\draw[-] (7.5,16.5) to (8.33,17.33);}
%\red{\draw[-] (9.5,18.5) to (10.33,19.33);}
%\red{\draw[-] (4.67,21.33) to (5.5,20.5);}
%\red{\draw[-] (2.67,15.33) to (3.5,14.5);}
%\red{\draw[-] (4.67,17.33) to (5.5,16.5);}
%\red{\draw[-] (8.67,17.33) to (9.5,16.5);}
%\red{\draw[-] (10.67,19.33) to (11.5,18.5);}
\node at (2,20.5) {$1$};
\node at (2,16.5) {$2$};
\node at (4.5,18.7) {$3$};
\node at (4,22.5) {$4$};
\node at (6,24.5) {$5$};
\node at (6.5,20.7) {$6$};
\node at (9,18.5) {$7$};
\node at (9,22.5) {$8$};
\node at (11,20.5) {$9$};

\node at (6.5,14) {$\iop(T)=2$};

\node at (14.5,19.5) {$\bullet$};
\node at (16.5,21.5) {$\circ$};
\node at (18.5,23.5) {$\bullet$};
\node at (14.5,15.5) {$\bullet$};
\node at (16.5,17.5) {$\circ$};
\node at (18.5,19.5) {$\circ$};
\node at (20.5,21.5) {$\circ$};
\node at (20.5,17.5) {$\bullet$};
\node at (22.5,19.5) {$\bullet$};
\draw[-] (14.7,19.7) to (16.33,21.33);
\draw[-] (16.7,21.7) to (18.33,23.33);
\draw[-] (14.7,15.7) to (16.33,17.33);
\draw[-] (18.7,19.7) to (20.33,21.33);
\draw[-] (14.67,19.33) to (16.3,17.7);
\draw[-] (18.67,23.33) to (20.3,21.7);
\draw[-] (18.67,19.33) to (20.3,17.7);
\draw[-] (20.67,21.33) to (22.3,19.7);
\node at (14,20.5) {$9$};
\node at (14,16.5) {$8$};
\node at (16.5,18.7) {$7$};
\node at (16,22.5) {$6$};
\node at (18,24.5) {$5$};
\node at (18.5,20.7) {$4$};
\node at (21,18.5) {$3$};
\node at (21,22.5) {$2$};
\node at (23,20.5) {$1$};

\node at (18.5,14) {$\riop(T)=1$};

\node at (26.5,19.5) {$\bullet$};
\node at (28.5,21.5) {$\circ$};
\node at (30.5,23.5) {$\bullet$};
\node at (26.5,15.5) {$\bullet$};
\node at (28.5,17.5) {$\circ$};
\node at (30.5,19.5) {$\circ$};
\node at (32.5,21.5) {$\circ$};
\node at (32.5,17.5) {$\bullet$};
\node at (34.5,19.5) {$\bullet$};
\draw[-] (26.7,19.7) to (28.33,21.33);
\draw[-] (28.7,21.7) to (30.33,23.33);
\draw[-] (26.7,15.7) to (28.33,17.33);
\draw[-] (30.7,19.7) to (32.33,21.33);
\draw[-] (26.67,19.33) to (28.3,17.7);
\draw[-] (30.67,23.33) to (32.3,21.7);
\draw[-] (30.67,19.33) to (32.3,17.7);
\draw[-] (32.67,21.33) to (34.3,19.7);
\node at (26,20.5) {$3$};
\node at (26,16.5) {$5$};
\node at (28.5,18.7) {$4$};
\node at (28,22.5) {$2$};
\node at (30,24.5) {$1$};
\node at (30.5,20.7) {$7$};
\node at (33,18.5) {$8$};
\node at (33,22.5) {$6$};
\node at (35,20.5) {$9$};

\node at (30.5,14) {$\top(T)=1$};

\node at (38.5,19.5) {$\bullet$};
\node at (40.5,21.5) {$\circ$};
\node at (42.5,23.5) {$\bullet$};
\node at (38.5,15.5) {$\bullet$};
\node at (40.5,17.5) {$\circ$};
\node at (42.5,19.5) {$\circ$};
\node at (44.5,21.5) {$\circ$};
\node at (44.5,17.5) {$\bullet$};
\node at (46.5,19.5) {$\bullet$};
\draw[-] (38.7,19.7) to (40.33,21.33);
\draw[-] (40.7,21.7) to (42.33,23.33);
\draw[-] (38.7,15.7) to (40.33,17.33);
\draw[-] (42.7,19.7) to (44.33,21.33);
\draw[-] (38.67,19.33) to (40.3,17.7);
\draw[-] (42.67,23.33) to (44.3,21.7);
\draw[-] (42.67,19.33) to (44.3,17.7);
\draw[-] (44.67,21.33) to (46.3,19.7);
\node at (38,20.5) {$7$};
\node at (38,16.5) {$9$};
\node at (40.5,18.7) {$8$};
\node at (40,22.5) {$6$};
\node at (42,24.5) {$1$};
\node at (42.5,20.7) {$4$};
\node at (45,18.5) {$5$};
\node at (45,22.5) {$2$};
\node at (47,20.5) {$3$};

\node at (42.5,14) {$\rtop(T)=1$};

%%%%%%%%%%%%%%%%%%%%%%%%%%%%%%%%%%%%%%%%%%%%%%%%%%%%%%%%%%%%
%%%%%%%%%%%%%%%%%%%%%%%%%%%%%%%%%%%%%%%%%%%%%%%%%%%%%%%%%%%%

\node at (2.5,6.5) {$\bullet$};
\node at (4.5,8.5) {$\circ$};
\node at (6.5,10.5) {$\bullet$};
\node at (2.5,2.5) {$\bullet$};
\node at (4.5,4.5) {$\circ$};
\node at (6.5,6.5) {$\circ$};
\node at (8.5,8.5) {$\circ$};
\node at (8.5,4.5) {$\bullet$};
\node at (10.5,6.5) {$\bullet$};
\draw[-] (2.7,6.7) to (4.33,8.33);
\draw[-] (4.7,8.7) to (6.33,10.33);
\draw[-] (2.7,2.7) to (4.33,4.33);
\draw[-] (6.7,6.7) to (8.33,8.33);
\draw[-] (2.67,6.33) to (4.3,4.7);
\draw[-] (6.67,10.33) to (8.3,8.7);
\draw[-] (6.67,6.33) to (8.3,4.7);
\draw[-] (8.67,8.33) to (10.3,6.7);
\node at (2,7.5) {$3$};
\node at (2,3.5) {$1$};
\node at (4.5,5.7) {$2$};
\node at (4,9.5) {$4$};
\node at (6,11.5) {$9$};
\node at (6.5,7.7) {$6$};
\node at (9,5.5) {$5$};
\node at (9,9.5) {$8$};
\node at (11,7.5) {$7$};

\node at (6.5,1) {$\pop(T)=1$};

\node at (14.5,6.5) {$\bullet$};
\node at (16.5,8.5) {$\circ$};
\node at (18.5,10.5) {$\bullet$};
\node at (14.5,2.5) {$\bullet$};
\node at (16.5,4.5) {$\circ$};
\node at (18.5,6.5) {$\circ$};
\node at (20.5,8.5) {$\circ$};
\node at (20.5,4.5) {$\bullet$};
\node at (22.5,6.5) {$\bullet$};
\draw[-] (14.7,6.7) to (16.33,8.33);
\draw[-] (16.7,8.7) to (18.33,10.33);
\draw[-] (14.7,2.7) to (16.33,4.33);
\draw[-] (18.7,6.7) to (20.33,8.33);
\draw[-] (14.67,6.33) to (16.3,4.7);
\draw[-] (18.67,10.33) to (20.3,8.7);
\draw[-] (18.67,6.33) to (20.3,4.7);
\draw[-] (20.67,8.33) to (22.3,6.7);
\node at (14,7.5) {$7$};
\node at (14,3.5) {$5$};
\node at (16.5,5.7) {$6$};
\node at (16,9.5) {$8$};
\node at (18,11.5) {$9$};
\node at (18.5,7.7) {$3$};
\node at (21,5.5) {$2$};
\node at (21,9.5) {$4$};
\node at (23,7.5) {$1$};

\node at (18.5,1) {$\rpop(T)=2$};

\node at (26.5,6.5) {$\bullet$};
\node at (28.5,8.5) {$\circ$};
\node at (30.5,10.5) {$\bullet$};
\node at (26.5,2.5) {$\bullet$};
\node at (28.5,4.5) {$\circ$};
\node at (30.5,6.5) {$\circ$};
\node at (32.5,8.5) {$\circ$};
\node at (32.5,4.5) {$\bullet$};
\node at (34.5,6.5) {$\bullet$};
\draw[-] (26.7,6.7) to (28.33,8.33);
\draw[-] (28.7,8.7) to (30.33,10.33);
\draw[-] (26.7,2.7) to (28.33,4.33);
\draw[-] (30.7,6.7) to (32.33,8.33);
\draw[-] (26.67,6.33) to (28.3,4.7);
\draw[-] (30.67,10.33) to (32.3,8.7);
\draw[-] (30.67,6.33) to (32.3,4.7);
\draw[-] (32.67,8.33) to (34.3,6.7);
\node at (26,7.5) {$4$};
\node at (26,3.5) {$9$};
\node at (28.5,5.7) {$7$};
\node at (28,9.5) {$2$};
\node at (30,11.5) {$1$};
\node at (30.5,7.7) {$5$};
\node at (33,5.5) {$8$};
\node at (33,9.5) {$3$};
\node at (35,7.5) {$6$};

\node at (30.5,1) {$\lop(T)=1$};

\node at (38.5,6.5) {$\bullet$};
\node at (40.5,8.5) {$\circ$};
\node at (42.5,10.5) {$\bullet$};
\node at (38.5,2.5) {$\bullet$};
\node at (40.5,4.5) {$\circ$};
\node at (42.5,6.5) {$\circ$};
\node at (44.5,8.5) {$\circ$};
\node at (44.5,4.5) {$\bullet$};
\node at (46.5,6.5) {$\bullet$};
\draw[-] (38.7,6.7) to (40.33,8.33);
\draw[-] (40.7,8.7) to (42.33,10.33);
\draw[-] (38.7,2.7) to (40.33,4.33);
\draw[-] (42.7,6.7) to (44.33,8.33);
\draw[-] (38.67,6.33) to (40.3,4.7);
\draw[-] (42.67,10.33) to (44.3,8.7);
\draw[-] (42.67,6.33) to (44.3,4.7);
\draw[-] (44.67,8.33) to (46.3,6.7);
\node at (38,7.5) {$6$};
\node at (38,3.5) {$9$};
\node at (40.5,5.7) {$8$};
\node at (40,9.5) {$3$};
\node at (42,11.5) {$1$};
\node at (42.5,7.7) {$5$};
\node at (45,5.5) {$7$};
\node at (45,9.5) {$2$};
\node at (47,7.5) {$4$};

\node at (42.5,1) {$\rlop(T)=1$};
\end{tikzpicture}
\caption{Eight types of traversal for a di-sk tree $T$ and its associated statistics}
\label{8traversal}
\end{figure}

We use $\st_1\sim_{S}\st_2$ to indicate that the two statistics $\st_1$ and $\st_2$ are equidistributed over the set $S$. For instance, Theorem~\ref{thm:sep:sym} implies in particular that $\iar\sim_{\S_n(2413,3142)}\comp$, or equivalently upon applying the mapping $\eta$, we have $\iop\sim_{\DT_n}\top$. We simply write $\st_1\sim\st_2$ when the set $S$ is clear from the context.

Our first result in this section classifies the corresponding eight statistics in Tab.~\ref{eight-tree-tra} into three equidistribution classes.

% \begin{enumerate}
% 	\item Inorder (Left--Root--Right). Denote the number of initial $\oplus$-nodes (by inorder) in $T$ as $\iop(T)$.
% 	\item Right inorder (Right--Root--Left). Denote the number of initial $\oplus$-nodes (by right inorder) in $T$ as $\riop(T)$.
% 	\item Preorder (Root--Left--Right). Denote the number of initial $\oplus$-nodes (by preorder) in $T$ as $\top(T)$.
% 	\item Right preorder (Root--Right--Left). Denote the number of initial $\oplus$-nodes (by right preorder) in $T$ as $\rtop(T)$.
% 	\item Postorder (Left--Right--Root). Denote the number of initial $\oplus$-nodes (by postorder) in $T$ as $\pop(T)$.
% 	\item Right postorder (Right--Left--Root). Denote the number of initial $\oplus$-nodes (by right postorder) in $T$ as $\rpop(T)$.
% 	\item Level order (Left--Right--Next level). Denote the number of initial $\oplus$-nodes (by level order) in $T$ as $\lop(T)$.
% 	\item Right level order (Right--Left--Next level). Denote the number of initial $\oplus$-nodes (by right level order) in $T$ as $\rlop(T)$.
% \end{enumerate}

\begin{theorem}\label{thm:transver}
Over the set of di-sk trees $\DT_n$, the distributions of the eight statistics in Tab.~\ref{eight-tree-tra} group into three classes. Namely, we have two classes
\begin{align}
&\rtop=\rlop, \text{ and}\label{two equal}\\
&\riop\sim\iop\sim\top\sim\pop\sim\rpop, \label{five equal}
\end{align}
where $\st_1=\st_2$ means that $\st_1(T)=\st_2(T)$ for each $T\in\DT_n$. The distribution of $\lop$ over $\DT_n$ is different from the other seven statistics for $n\ge 5$.
\end{theorem}
\begin{proof}
Firstly, it can be quickly checked that the eight statistics have the same distribution among the two di-sk trees in $\DT_2$ and the six di-sk trees in $\DT_3$. We next show that $\rtop=\rlop$ by induction on $n$. Take any di-sk tree $T\in\DT_n$, if the root of $T$ is an $\ominus$-node, then $\rtop(T)=\rlop(T)=0$. Otherwise we can write $T=(T_1,\oplus,T_2)$, with $T_1$ and $T_2$ being the left subtree and the right subtree of $T$, respectively. Now if $T_2$ is nonempty, then its root must be an $\ominus$-node since $T$ is a di-sk tree, thus we have $\rtop(T)=\rlop(T)=1$. If $T_2$ is empty, then we have
$\rtop(T)=1+\rtop(T_1)=1+\rlop(T_1)=\rlop(T)$ as well. We have finished the proof of \eqref{two equal}.

Next, we prove $\riop\sim\iop$ via a bijection. Recall the elementary operation {\em reverse-complement} on permutations:
\begin{align*}
\pi=\pi_1\cdots\pi_n &\mapsto \rc(\pi):=(n+1-\pi_n)\cdots (n+1-\pi_1).
\end{align*}
It then suffices to note that for each $T\in\DT_n$,
\begin{align*}
\riop(T)=\iop(\eta\circ\rc\circ\eta^{-1}(T)),
\end{align*}
and the fact that $\S_n(2413,3142)$ is closed under the reverse-complement map.

Since we have already proved the symmetry of $(\iop,\top)$ in Theorem~\ref{thm:di-sk}, it remains to show the two equidistributions $\top\sim\rpop$ and $\pop\sim\rpop$ 
to finish the proof of~\eqref{five equal}. The first one follows from the symmetry of $(\rpop,\top)$ that will be proved in Theorem~\ref{thm:rpop} using generating functions, while the second one is proved 
 via a recursively constructed bijection $\theta$ that we define next. Indeed, for each $n\ge 1$, we construct a bijection $\theta:\DT_n\rightarrow\DT_n$  satisfying 
\begin{align}\label{bij:pop=rpop}
(\pop,\rpop) T=(\rpop,\pop) \theta(T),
\end{align} 
for every $T\in\DT_n$. Let $\emptyset$ be the empty tree and we set $\theta(\emptyset)=\emptyset$. For a di-sk tree $T \in \mathfrak{DT}_n$ with $n \geq 2$, suppose $T=(T_1,r,T_2)$, where $r=\oplus$ or $\ominus$ is the root of $T$, while $T_1$ (resp.~$T_2$) is the left (resp.~right) subtree of $T$. We consider the following three cases.
\begin{itemize}
	\item If $T_2=\emptyset$, then we define $\theta(T)=(\theta(T_1),r,\emptyset)$.
	\item If $T_2\neq\emptyset$ and $T_1=\emptyset$, then we let $\theta(T)$ be the unique di-sk tree in $\DT_n$ having no left subtree and $\theta(T_2)$ as its right subtree.
	\item Otherwise $T_1\neq\emptyset$ and $T_2\neq\emptyset$, then we let $\theta(T)$ be the unique di-sk tree in $\DT_n$ having $\theta(T_2)$ and $\theta(T_1)$ as its left and right subtrees, respectively.
\end{itemize}
In all three cases, we can verify that~\eqref{bij:pop=rpop} holds true assuming it has been proved for trees with fewer nodes. It should also be clear how to invert the map $\theta$. The proof is now completed.
\end{proof}

For a di-sk tree $T$, let $\omi(T)$ be the number of $\ominus$-nodes in $T$. 
\begin{theorem}\label{thm:rpop}
For $n\geq1$, the two triples $(\omi,\rpop,\top)$ and $(\omi,\top,\rpop)$ have the same distribution over $\DT_n$.
In particular, the pair $(\rpop,\top)$ is symmetric over $\DT_n$.
\end{theorem}

\begin{proof}
We will prove the equidistribution by using generating functions.
Let us consider the generating function
$$
S(x,y)=S(t,x,y;z):=1+\sum_{n\ge 1}z^n\sum_{T\in\DT_{n+1}}t^{\omi(T)}x^{\rpop(T)}y^{\top(T)}=1+S^{\oplus}(x,y)+S^{\ominus}(x,y),
$$
where $S^{\oplus}(x,y)$ (resp.~$S^{\ominus}(x,y)$) is the generating function for nonempty di-sk trees whose root is an $\oplus$-node (resp.~$\ominus$-node). For the sake of simplicity, we set $S=S(1,1)$. By the work in~\cite{flz}, the function $S$ satisfies the algebraic equation
\begin{equation}\label{S:des}
S=tz^2S^3+tz^2S^2+(1+t)zS+1.
\end{equation}
On the other hand, by Theorem~\ref{thm:di-sk} the two pairs $(\omi,\top)$ and $(\omi,\iop)$ are equidistributed over di-sk trees. Thus, the pair $(\omi,\top)$ over di-sk trees has the same distribution as the pair $(\des,\iar-1)$ over separable permutations and it follows from~\cite[Eq.~(5.4)]{flw} that 
\begin{equation}\label{S:top}
S(1,y)=\frac{S}{1+(1-y)zS}.
\end{equation}

Let $T=(T_1,r,T_2)$ be a nonempty di-sk tree, with $r=\oplus$ or $\ominus$ being the label of its root, $T_1$ and $T_2$ being its left subtree and right subtree, respectively. We need to consider two cases:
\begin{itemize}
\item {\bf Case~1: $r=\oplus$}. We further distinguish two cases.
\begin{enumerate}
\item $T_2=\emptyset$. 
\begin{itemize}
\item[a)] $T_1$ is a  di-sk tree (possibly empty) without any $\ominus$-node. This case contributes to $S^{\oplus}(x,y)$ the enumerator 
$$\frac{zxy}{1-zxy}.$$
\item[b)] Otherwise, $T_1$ is a  di-sk tree with at least one $\ominus$-node. This case contributes to $S^{\oplus}(x,y)$ the enumerator 
$$zy\biggl(S(x,y)-\frac{1}{1-zxy}\biggr).$$
\end{itemize}
\item $T_2\neq\emptyset$. This case contributes to $S^{\oplus}(x,y)$ the enumerator
$$zyS(1,y)S^{\ominus}(x,1).$$
\end{enumerate}
\item {\bf Case~2: $r=\ominus$}. We further distinguish two cases.
\begin{enumerate}
\item $T_2$ is a  di-sk tree (possibly empty) without any $\ominus$-node. This case contributes to $S^{\ominus}(x,y)$ the enumerator
$$
\frac{tzS(x,1)}{1-zx}.
$$
\item $T_2$ is a  di-sk tree with at least one $\ominus$-node. This case contributes to $S^{\ominus}(x,y)$ the enumerator 
$$tzS\biggl(S^{\oplus}(x,1)-\frac{zx}{1-zx}\biggr).$$
\end{enumerate}
\end{itemize}
Combining all the above cases results in a system of functional equations
\begin{equation}\label{sys:rpop}
\begin{cases}
\,\, S(x,y)=1+S^{\oplus}(x,y)+S^{\ominus}(x,y),\\
\,\, S^{\oplus}(x,y)=\frac{zxy}{1-zxy}+zy\biggl(S(x,y)-\frac{1}{1-zxy}\biggr)+zyS(1,y)S^{\ominus}(x,1),\\
\,\, S^{\ominus}(x,y)=\frac{tzS(x,1)}{1-zx}+tzS\biggl(S^{\oplus}(x,1)-\frac{zx}{1-zx}\biggr).
\end{cases}
\end{equation}
Solving this system of equations for $y=1$ (using Maple) yields
\begin{align}
 S(x,1)&=-(tz^2S^2+tz^2S+z-1)/A,\label{S:rpop}\\
 S^{\oplus}(x,1)&=(txz^2(zx-1)S^2+tz(1-zx)^2S+x(1-z)(1-zx)+tz(1-x))/B,\label{Sop:rpop}\\
S^{\ominus}(x,1)&=(z-1)tz/B\label{Som:rpop},
\end{align}
where 
\begin{align*}
A&:=(txz^3-tz^2)S^2+(txz^3-2tz^2)S+xz^2-tz-xz-z+1,\\
B&:=(1-zx)(tz^2(1-zx)S^2+tz^2(2-zx)S+tz+zx+z-xz^2-1).
\end{align*}
It follows from~\eqref{S:rpop} and~\eqref{S:top} that 
\begin{align*}
S(x,1)-S(1,x)=\frac{(z-1)(tz^2S^3+tz^2S^2+(1+t)zS+1-S)}{A(xzS-zS-1)}=0,
\end{align*}
in view of~\eqref{S:des}.  Equivalently, 
\begin{equation}\label{sym:S}
S(x,1)=S(1,x)=\frac{S}{1+(1-x)zS}.
\end{equation}
 Substituting~\eqref{sym:S},~\eqref{Sop:rpop} and~\eqref{Som:rpop} into~\eqref{sys:rpop} and solving this system of equations (using Maple) gives an expression for $S(x,y)$, which is a rational function in $S,t,z,x$ and $y$ (too complicated  to be reported here). It turns out that there is a factor 
 $$tz^2S^3+tz^2S^2+(1+t)zS+1-S,$$ which is zero in view of~\eqref{S:des}, in the numerator of the   difference $S(x,y)-S(y,x)$. This proves $S(x,y)=S(y,x)$, as desired. 
\end{proof}

In Theorems~\ref{thm:di-sk},~\ref{thm:transver} and~\ref{thm:rpop}, we have already proved three symmetries over di-sk trees: 
\begin{align*}
&(\omi,\top,\iop)\sim(\omi,\iop,\top),\\
&(\pop,\rpop)\sim(\rpop,\pop),\\
& (\omi,\top,\rpop)\sim(\omi,\rpop,\top).
\end{align*}
Recall that our proof of $\riop\sim\iop$ is essentially the application of the reverse-complement operation composed with the bijection $\eta$. Note that this composed map preserves the statistics $\top$ and $\omi$. This means we get the following symmetry over di-sk trees for free:
\begin{align*}
&(\omi,\top,\riop)\sim(\omi,\riop,\top).
\end{align*} 
Among the remaining six pairs taken from the equidistributed quintuple $$\riop\sim\iop\sim\top\sim\pop\sim\rpop,$$ our calculations suggest the following three more symmetric pairs. 
\begin{conj}\label{conj:sym}
Over di-sk trees, the following three symmetries hold:
\begin{align*}
&(\omi,\riop,\rpop)\sim(\omi,\rpop,\riop),\\
&(\omi,\iop,\rpop)\sim(\omi,\rpop,\iop),\\
& (\riop,\pop)\sim(\pop,\riop).
\end{align*}
\end{conj}
Although Conjecture~\ref{conj:sym} may be proved by computing their generating functions similarly as the proof of Theorem~\ref{thm:rpop},  combinatorial involution proofs are preferred. 

\subsection{Initial $\ominus$-nodes in di-sk trees}

In the previous section, we have considered eight different types of  tree  traversal and their associated statistics for the number of initial $\oplus$-nodes in di-sk trees. One could also consider the statistics of the number of initial $\ominus$-nodes with respect to the  eight types of tree traversal; see Tab.~\ref{eight-tree-omi} for the notations of these eight associated statistics. Since $\oplus$-nodes and $\ominus$-nodes are symmetry in di-sk trees, for  a fixed type of traversal,  the statistic of the number of initial $\oplus$-nodes is equidistributed with the statistic of  the number of initial $\ominus$-nodes. For instance, we have $\iop\sim\iom$ over $\DT_n$. Thus, by Theorem~\ref{thm:transver} we have five new Comtet statistics over di-sk trees
$$
\riom\sim\iom\sim\tom\sim\pom\sim\rpom.
$$
It would be interesting to investigate systematically the joint distribution for one of the five Comtet statistics above and one in~\eqref{five equal}. As one example, in the rest of this paper,  we aim to prove combinatorially that the Comtet pair $(\top,\iom)$ is symmetric over di-sk trees.

\begin{table}[bp]\caption{Eight types of tree traversal for a di-sk tree $T$ and their associated statistics for the number of initial $\ominus$-nodes}\label{eight-tree-omi}
\centering
\begin{tabular}{|c||c|c|c|c|}
\hline
Name & inorder & right inorder & preorder & right preorder \\
\hline
Rule & Left--Root--Right & Right--Root--Left & Root--Left--Right & Root--Right--Left \\
\hline
Stat & $\iom(T)$ & $\riom(T)$ & $\tom(T)$ & $\rtom(T)$ \\
\hline
Name & postorder & right postorder & level order & right level order \\
\hline
Rule & Left--Right--Root & Right--Left--Root & Left--Right--Next level & Right--Left--Next level \\
\hline
Stat & $\pom(T)$ & $\rpom(T)$ & $\lom(T)$ & $\rlom(T)$ \\
\hline
\end{tabular}
\end{table}

For $n\geq1$ and  $0\leq k,l\leq n-1$, let us consider the set 
\begin{equation}\label{omi-top}
\DT_n^{(k,l)}:=\{T\in\DT_n: \top(T)=k, \iom(T)=l\}
\end{equation}
and denote its cardinality by $s_n^{(k,l)}$. Interestingly, it turns out that the $n\times n$ matrix $M_n^{\top,\iom}$, with entry $s_n^{(i-1,j-1)}$ in row $i$ and column $j$, is upper anti-triangular and Hankel. The first values of $M_n^{\top,\iom}$ are 
$$
\begin{bmatrix}\label{tri:Sch2}
0& 1 \\
1 & 0 
\end{bmatrix},
\begin{bmatrix}\label{tri:Sch3}
1 & 1 & 1 \\
1 & 1  & 0\\
1 & 0 & 0
\end{bmatrix},
\begin{bmatrix}\label{tri:Sch4}
4 & 4 & 2&1 \\
4 & 2  & 1&0\\
2 & 1 & 0&0\\
1&0&0&0
\end{bmatrix},
\begin{bmatrix}\label{tri:Sch5}
17 & 16 & 8&3 &1\\
16 & 8  & 3&1&0\\
8 & 3& 1&0&0\\
3&1&0&0&0\\
1&0&0&0&0
\end{bmatrix},
\begin{bmatrix}\label{tri:Sch}
76 & 69 & 34&13 &4&1\\
69 & 34 & 13&4&1&0\\
34 & 13& 4&1&0&0\\
13&4&1&0&0&0\\
4&1&0&0&0&0\\
1&0&0&0&0&0
\end{bmatrix}.
$$
The upper anti-triangular property of $M_n^{\top,\iom}$ follows from the simple fact that $\top(T)+\iom(T)\leq n-1$ for each $T\in\DT_n$, while the Hankel property of $M_n^{\top,\iom}$ is a consequence of the following result.

\begin{theorem}\label{thm:bij2}
Let $\DT_n^{(k,l)}$ be defined in~\eqref{omi-top}. If $k\geq1$, then there exists a bijection $\psi:\DT_n^{(k,l)}\rightarrow \DT_n^{(k-1,l+1)}$. Consequently, the pair $(\comp,\idr)$ of double Comtet statistics is symmetric over $\S_n(2413,3142)$, where $\idr(\pi)$ denotes the length of the initial descending  run of a permutation $\pi$.
\end{theorem}

In order to facilitate our construction of $\psi$, we define three basic operations, called {\em conjugation, insertion}, and {\em extraction}, for di-sk trees. 

\begin{Def}[Conjugation]
Given a di-sk tree $T\in\DT_n$, we reverse the labels of all of its nodes, i.e., $\oplus$-nodes become $\ominus$-nodes and $\ominus$-nodes become $\oplus$-nodes. This yields a new di-sk tree in $\DT_n$ that we call the {\em conjugate} of $T$. Denote by    $\widebar{T}$ the conjugate of $T$.
\end{Def}

% We call the path from the root of $T$ to the iroot of $T$ the \emph{spine} of $T$. If two nodes $u$ and $v$ in a tree are connected by a path consisted of left edges only, we say $u$ and $v$ are {\em on the same (left) branch}. For example, the root and the iroot of any tree are on the same branch. We will often write $T=(T_1,v,T_2)$, with $v$ (resp.~$T_1$, $T_2$) being the root (resp.~left subtree, right subtree) of $T$, respectively. Both $T_1$ and $T_2$ could be empty. 

\begin{figure}
\begin{tikzpicture}[scale=0.3]
%%%%%%%%% extraction 
\node at (9.5,8.5) {$\circ$};
\node at (7.5,6.5) {$\circ$};
\node at (5.5,4.5) {$\circ$};
\node at (3.5,2.5) {$\circ$};

\draw[-,dashed] (1,0) to (3.33,2.33);
\draw[-] (3.7,2.7) to (5.33,4.33);
\draw[-,dashed] (5.7,4.7) to (7.33,6.33);
\draw[-] (7.7,6.7) to (9.33,8.33);
\draw[-,dashed] (9.7,8.7) to (12,11);

\draw[-,dashed] (3.67,2.33) to (6,0);
\draw[-,dashed] (5.67,4.33) to (8,2);
\draw[-,dashed] (7.67,6.33) to (10,4);
\draw[-,dashed] (9.67,8.33) to (12,6);
\draw[-,dashed] (7,11) to (9.3,8.7);

\node at (2.75,3.25) {$a$};
\node at (4.75,5.25) {$c$};
\node at (6.75,7.25) {$d$};
\node at (9.5,10) {$b$};

\node at (14.5,5.5) {$\stackrel{\text{extract}}{\longrightarrow}$};
%%%%%%%%%%%%%%%%%%%%%%%%%%%%%%%%%%%%%%%
\node at (20.5,6.5) {$\circ$};
\node at (18.5,4.5) {$\circ$};

\draw[-,dashed] (18.7,4.7) to (20.33,6.33);

\draw[-,dashed] (18.67,4.33) to (21,2);
\draw[-,dashed] (20.67,6.33) to (23,4);

\node at (17.75,5.25) {$c$};
\node at (19.75,7.25) {$d$};

\node at (26.5,5.5) {$\Longrightarrow$};
%%%%%%%%%%%%%%%%%%%%%%%%%%%%%%%%%%%%%%%%%%
\node at (33.5,6.5) {$\circ$};
\node at (31.5,4.5) {$\circ$};

\draw[-,dashed] (29,2) to (31.33,4.33);
\draw[-] (31.7,4.7) to (33.33,6.33);
\draw[-,dashed] (33.7,6.7) to (36,9);

\draw[-,dashed] (31.67,4.33) to (34,2);
\draw[-,dashed] (33.67,6.33) to (36,4);
\draw[-,dashed] (31,9) to (33.3,6.7);

\node at (30.75,5.25) {$a$};
\node at (33.5,8) {$b$};

\node at (6,-1.3) {$T_3$};
\node at (19.75,-1.3) {$T_2$};
\node at (32,-1.3) {$T_1=T_3\backslash T_2$};

%%%%%%%%%%%%%%%%%%%%%%%%%%%%%%%%%%%%%%%%%%%%%%%

%%%%%%%%%%%%% insertion
%%%%%%%%%%%%%%%%%%%%%%%%%%%%%%%%%%%%%%%%%%%%%%%
\node at (7.5,20.5) {$\circ$};
\node at (5.5,18.5) {$\circ$};

\draw[-,dashed] (3,16) to (5.33,18.33);
\draw[-] (5.7,18.7) to (7.33,20.33);
\draw[-,dashed] (7.7,20.7) to (10,23);

\draw[-,dashed] (5.67,18.33) to (8,16);
\draw[-,dashed] (7.67,20.33) to (10,18);
\draw[-,dashed] (5,23) to (7.3,20.7);

\node at (4.75,19.25) {$a$};
\node at (7.5,22) {$b$};
\node at (6,13) {$T_1$};
\node at (19.75,13) {$T_2$};
\node at (33,13) {$T_3=T_1/T_2(a,b)$};

\node at (14.5,19.5) {$\stackrel{\text{insert}}{\longleftarrow}$};
%%%%%%%%%%%%%%%%%%%%%%%%%%%%%%%%%%%%%%%%%%%%%%%%
\node at (20.5,20.5) {$\circ$};
\node at (18.5,18.5) {$\circ$};

\draw[-,dashed] (18.7,18.7) to (20.33,20.33);

\draw[-,dashed] (18.67,18.33) to (21,16);
\draw[-,dashed] (20.67,20.33) to (23,18);

\node at (17.75,19.25) {$c$};
\node at (19.75,21.25) {$d$};

\node at (26.5,19.5) {$\Longrightarrow$};
%%%%%%%%%%%%%%%%%%%%%%%%%%%%%%%%%%%%%%%%%%%%%%%%%
\node at (35.5,22.5) {$\circ$};
\node at (33.5,20.5) {$\circ$};
\node at (31.5,18.5) {$\circ$};
\node at (29.5,16.5) {$\circ$};

\draw[-,dashed] (27,14) to (29.33,16.33);
\draw[-] (29.7,16.7) to (31.33,18.33);
\draw[-,dashed] (31.7,18.7) to (33.33,20.33);
\draw[-] (33.7,20.7) to (35.33,22.33);
\draw[-,dashed] (35.7,22.7) to (38,25);

\draw[-,dashed] (29.67,16.33) to (32,14);
\draw[-,dashed] (31.67,18.33) to (34,16);
\draw[-,dashed] (33.67,20.33) to (36,18);
\draw[-,dashed] (35.67,22.33) to (38,20);
\draw[-,dashed] (33,25) to (35.3,22.7);

\node at (28.75,17.25) {$a$};
\node at (30.75,19.25) {$c$};
\node at (32.75,21.25) {$d$};
\node at (35.5,24) {$b$};

\end{tikzpicture}
\caption{Insertion and extraction.}
\label{fig:ins-ext}
\end{figure}
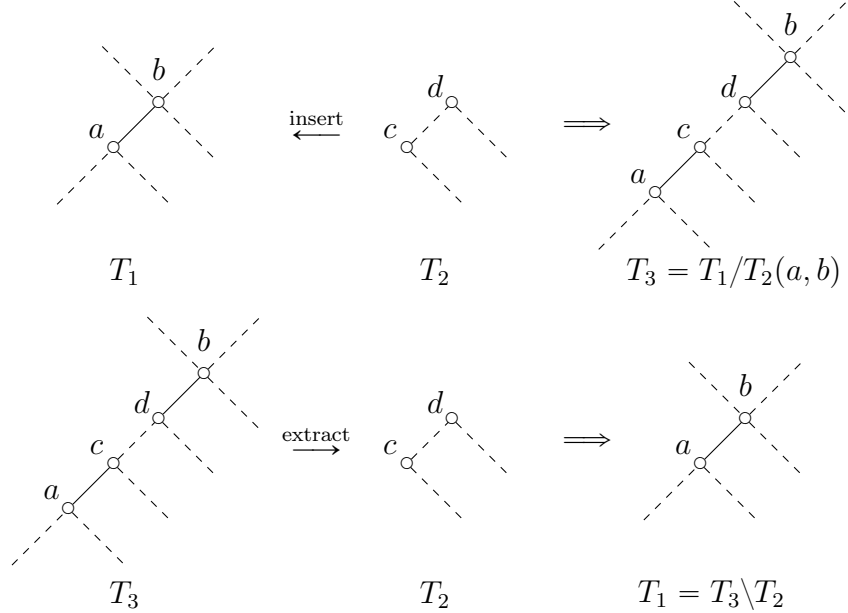

\begin{Def}[Insertion and extraction]
Given two trees $T_1,T_2$, suppose $a,b$ are two nodes of $T_1$ such that $a$ is the left child of $b$, and suppose $c$ (resp.~$d$) is the iroot (resp.~root) of $T_2$. We derive a new tree, say $T_3$, by first deleting the edge between $a$ and $b$, then attaching the subtree rooted at $a$ to the left of $c$, and attaching the subtree rooted at $d$ (now contains $a$) to the left of $b$. This operation is illustrated in Fig.~\ref{fig:ins-ext} (wherein $\circ$ does not indicate the label) and is called the {\em insertion of $T_2$ into $T_1$ at $a,b$}. We denote $T_3=T_1/T_2(a,b)$. Conversely, suppose $T_3$ and $T_2$ (with root $d$ and iroot $c$) are two trees such that $T_2$ can be embedded in $T_3$ satisfying
\begin{itemize}
	\item[i.] $d$ is the left child of certain node $b\in T_3$;
	\item[ii.] if we denote the left child of $c$ in $T_3$ as $a$, then the two edges $ca$ and $bd$ are the only edges connected to $T_2$ but not contained in $T_2$.
\end{itemize}
We derive a new tree, say $T_1$, from $T_3$ by first deleting the edges $ca$ and $bd$, then connecting $a$ and $b$ with a left edge. This operation is also illustrated in Fig.~\ref{fig:ins-ext} and is called the {\em extraction of $T_2$ from $T_3$}. We denote $T_1=T_3\backslash T_2$.
\end{Def}

\begin{remark}
To make the above definition applicable in more situations, we allow either $a$ or $b$ to be the empty node $\emptyset$. We explain here the meaning of such special cases for the insertion, while the extraction should be understood similarly. The meaning of $T_1/T_2(\emptyset,b)$ should be clear. For the case of $b=\emptyset$, $a$ must then be the root of $T_1$, and $T_1/T_2(a,\emptyset)$ is the tree rooted at $d$, with $a$ attached to $c$ by a left edge.
\end{remark}

Now we are ready for the proof of Theorem~\ref{thm:bij2}. 

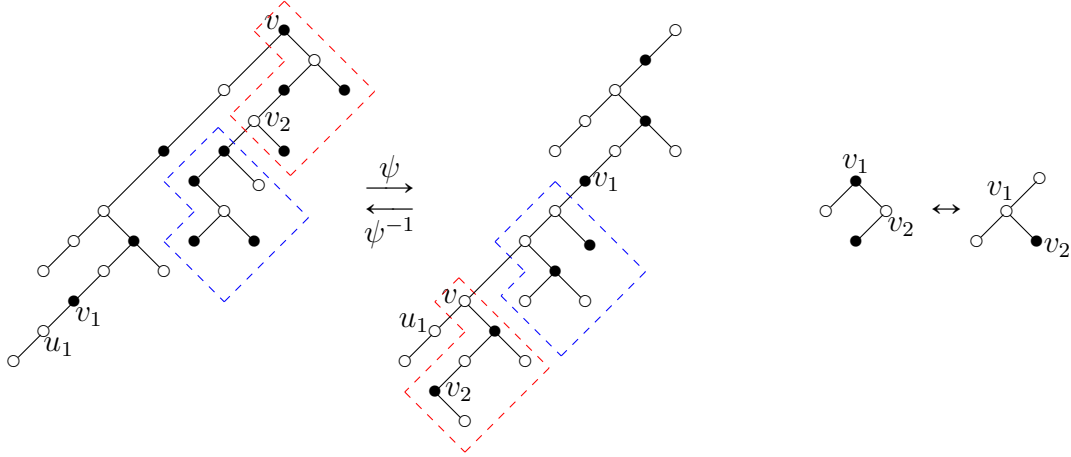
\begin{figure}
\begin{tikzpicture}[scale=0.4]
%%% p %%%%
%% edges
\draw[-] (2.15,3.15) to (2.9,3.9);
\draw[-] (3.15,4.15) to (4.9,5.9);
\draw[-] (5.15,6.15) to (6,7);
\draw[-] (8.15,7.15) to (8.9,7.9);
\draw[-] (8.15,9.15) to (8.9,9.9);
\draw[-] (9.15,10.15) to (9.9,10.9);
\draw[-]  (10.15,11.15) to (11,12);
\draw[-]  (11.9,12.9) to (11.15,12.15);
\draw[-] (3.15,6.15) to (3.9,6.9);
\draw[-] (4.15,7.15) to (4.9,7.9);
\draw[-] (5.15,8.15) to (8.9,11.9);
\draw[-] (9.15,12.15) to (10.9,13.9);

\draw[-] (6.85,6.15) to (5.1,7.9);
\draw[-] (10,7) to (9.1,7.9);
\draw[-] (8,9) to (8.85,8.15);
\draw[-] (10,9) to (9.1,9.9);
\draw[-]  (10.13,10.87) to (10.85,10.15);
\draw[-]  (12.1,12.9) to (12.85,12.15);
\draw[-]   (11,14) to (11.85,13.15);

%% labels
\node at (2,3) {$\circ$};
\node at (3,4) {$\circ$};
\node at (3.5,3.5) {$u_1$};
\node at (4,5) {$\bullet$};
\node at (4.5,4.5) {$v_1$};
\node at (5,6) {$\circ$};
\node at (6,7) {$\bullet$};
\node at (7,6) {$\circ$};
\node at (8,7) {$\bullet$};
\node at (9,8) {$\circ$};
\node at (10,7) {$\bullet$};
\node at (8,9) {$\bullet$};
\node at (9,10) {$\bullet$};
\node at (10.15,8.85) {$\circ$};
\node at (10,11) {$\circ$};
\node at (10.8,11) {$v_2$};
\node at (11,10) {$\bullet$};
\node at (11,12) {$\bullet$};
\node at (12,13) {$\circ$};
\node at (13,12) {$\bullet$};
\node at (11,14) {$\bullet$};
\node at (10.6,14.3) {$v$};
\node at (9,12) {$\circ$};
\node at (7,10) {$\bullet$};
\node at (5,8) {$\circ$};
\node at (4,7) {$\circ$};
\node at (3,6) {$\circ$};

%% edges1
\red{\draw[-,dashed] (11,15) to (10,14);}
\red{\draw[-,dashed] (10,14) to (11,13);}
\red{\draw[-,dashed] (11,13) to (9.2,11.2);}
\red{\draw[-,dashed] (9.2,11.2) to (11.2,9.2);}
\red{\draw[-,dashed] (11.2,9.2) to (14,12);}
\red{\draw[-,dashed] (11,15) to (14,12);}

\blue{\draw[-,dashed] (8.8,10.8) to (11.8,7.8);}
\blue{\draw[-,dashed] (8.8,10.8) to (7,9);}
\blue{\draw[-,dashed] (7,9) to (8,8);}
\blue{\draw[-,dashed] (8,8) to (7,7);}
\blue{\draw[-,dashed] (7,7) to (9,5);}
\blue{\draw[-,dashed] (9,5) to (11.8,7.8);}

%% edges1
\red{\draw[-,dashed] (16.8,5.8) to (19.8,2.8);}
\red{\draw[-,dashed] (16.8,5.8) to (16,5);}
\red{\draw[-,dashed] (16,5) to (17,4);}
\red{\draw[-,dashed] (17,4) to (15,2);}
\red{\draw[-,dashed] (15,2) to (17,0);}
\red{\draw[-,dashed] (17,0) to (19.8,2.8);}

\blue{\draw[-,dashed] (20,9) to (18,7);}
\blue{\draw[-,dashed] (18,7) to (19,6);}
\blue{\draw[-,dashed] (19,6) to (18.2,5.2);}
\blue{\draw[-,dashed] (18.2,5.2) to (20.2,3.2);}
\blue{\draw[-,dashed] (20.2,3.2) to (23,6);}
\blue{\draw[-,dashed] (20,9) to (23,6);}

\node at (14.5,8.7) {$\longrightarrow$};
\node at (14.5,9.4) {$\psi$};
\node at (14.5,8) {$\longleftarrow$};
\node at (14.5,7.3) {$\psi^{-1}$};
%%% q %%%%
%% edges
\draw[-] (16,2) to (16.9,2.9);
\draw[-] (17.15,3.15) to (18,4);
\draw[-] (19.15,5.15) to (20,6);
\draw[-] (15.15,3.15) to (15.9,3.9);
\draw[-] (16.15,4.15) to (16.9,4.9);
\draw[-] (17.15,5.15) to (18.9,6.9);
\draw[-] (19.15,7.15) to (19.9,7.9);
\draw[-] (20.15,8.15) to (21.9,9.9);
\draw[-] (22.15,10.15) to (23,11);
\draw[-] (20.15,10.15) to (20.9,10.9);
\draw[-] (21.15,11.15) to (21.9,11.9);
\draw[-] (22.15,12.15) to (23.9,13.9);

\draw[-] (16.85,1.15) to (16,2);
\draw[-] (18.85,3.15) to (17.1,4.9);
\draw[-] (20.85,5.15) to (19.1,6.9);
\draw[-] (21,7) to (20.1,7.9);
\draw[-] (23.85,10.15) to (22.1,11.9);

%% labels
\node at (24,14) {$\circ$};
\node at (23,13) {$\bullet$};
\node at (22,12) {$\circ$};
\node at (21,11) {$\circ$};
\node at (20,10) {$\circ$};
\node at (23,11) {$\bullet$};
\node at (22,10) {$\circ$};
\node at (21,9) {$\bullet$};
\node at (21.7,9) {$v_1$};
\node at (20,8) {$\circ$};
\node at (19,7) {$\circ$};
\node at (17,5) {$\circ$};
\node at (16.5,5.2) {$v$};
\node at (16,4) {$\circ$};
\node at (15.3,4.3) {$u_1$};
\node at (15,3) {$\circ$};
\node at (24,10) {$\circ$};
\node at (21.15,6.85) {$\bullet$};
\node at (20,6) {$\bullet$};
\node at (19,5) {$\circ$};
\node at (18,4) {$\bullet$};
\node at (17,3) {$\circ$};
\node at (16,2) {$\bullet$};
\node at (16.8,2) {$v_2$};
\node at (21,5) {$\circ$};
\node at (19,3) {$\circ$};
\node at (17,1) {$\circ$};

%%% p1 %%%%
%% edges
\draw[-] (30,7) to (30.9,7.9);
\draw[-] (29.15,8.15) to (30,9);

\draw[-] (30.85,8.15) to (30,9);

%% labels
\node at (30,9) {$\bullet$};\node at (30,9.6) {$v_1$};
\node at (29,8) {$\circ$};
\node at (31,8) {$\circ$};\node at (31.5,7.5) {$v_2$};
\node at (30,7) {$\bullet$};

\node at (33,8) {$\leftrightarrow$};
%%% q1 %%%%
%% edges
\draw[-] (34.15,7.15) to (34.9,7.9);
\draw[-] (35.15,8.15) to (36,9);

\draw[-] (36,7) to (35.12,7.88);

%% labels
\node at (36.1,9.1) {$\circ$};
\node at (35,8) {$\circ$};\node at (34.8,8.7) {$v_1$};
\node at (34,7) {$\circ$};
\node at (36,7) {$\bullet$};\node at (36.7,6.8) {$v_2$};

\end{tikzpicture}
\caption{Two examples of $\psi$: a general one and a small one \label{bij:psi}}
\end{figure}

\begin{proof}[{\bf Proof of Theorem~\ref{thm:bij2}}]
For each tree $T\in\DT_n^{(k,l)}$, since $k\geq1$, we can write $T=(T_1,v,T_2)$ with $v=\oplus$. We perform the following {\em cut-and-paste} procedure to get $\psi(T)$: 
\begin{itemize}
\item[Step 1] Let $v_1$ be the first $\oplus$-node (by inorder) of $T$ and $u_1$ be its left child (if $v_1$ has no left child then set $u_1=\emptyset$). Now if $T_2=\emptyset$, set $T^*=(\emptyset,\ominus,\emptyset)$ and jump to Step 3.

\item[Step 2] Otherwise $T_2\neq\emptyset$ and the root of $T_2$ must be an $\ominus$-node since $T$ is a di-sk tree. Now denote the lowest $\ominus$-node on the spine of $T_2$ as $v_2$, and denote the left subtree (possibly empty) of $v_2$ as $T_3$. We set $T^*$ to be the {\bf conjugate} of the following tree
$$((T\backslash T_1)\backslash T_3)/T_3(v,\emptyset)$$
 % Let $T_2'$ be the left subtree of $v_2$ in $T_2$ and $\widetilde T_2$ be the remaining tree after cutting $T_2'$ from $T_2$. Let $\widetilde T_2(v)=(\emptyset,v,\widetilde T_2)$ be the tree with $v$ as root and $\widetilde T_2$ as right subtree of $v$. Then, construct  the tree $T_3$ by attaching $\widetilde T_2(v)$ at $v$ as left subtree to the lowest node of the spine of $T_2'$. Let $T_3^*$ be the di-sk tree  obtained from $T_3$ by changing all $\oplus$-nodes (resp.~$\ominus$-nodes) into $\ominus$-nodes (resp.~$\oplus$-nodes). 

\item[Step 3] If $v_1$ coincides with $v$ (see a small example on right side of Fig.~\ref{bij:disk}), then take $\psi(T)=T_1/T^*(u_1,\emptyset)$. Otherwise, we take $\psi(T)=T_1/T^*(u_1,v_1)$.
% \item[Step 4] Let $T_1'$ be the left subtree of $v_1$ in $T_1$. Now cut $T_1'$ from $T_1$, then attach $T_3^*$ as new left subtree of $v_1$ and finally attach $T_1'$ back at $v$, which is the lowest node of the spine of $T_3^*$, as left subtree of $v$. The resulting di-sk tree is $\psi(T)$. 
\end{itemize}
It should be clear from our construction that $\psi(T)\in\DT_n^{(k-1,l+1)}$ indeed. See Fig.~\ref{bij:disk}  (on the left side) for an illustration of $\psi$. 

It remains to show that $\psi$ is invertible. We will construct its inverse $\psi^{-1}$ explicitly. Given a di-sk tree $T\in\DT_n^{(k-1,l+1)}$, we perform the following {\em inverse  cut-and-paste} procedure to get $\psi^{-1}(T)$: 
\begin{itemize}
\item[Step 1] 
Since $T\in\DT_n^{(k-1,l+1)}$, we have $\iom(T)=l+1$, so we can find the $(l+1)$-th (by inorder) node, which is an $\ominus$-node and denoted $v$. Let $v_2$ be the $(l+2)$-th node (by inorder), which is an $\oplus$-node. In the special case of $l+1=n-1$, we set $v_2=\emptyset$. If $v_2$ is not a descendent of $v$ (including the case $v_2=\emptyset$), set $T^*=(\emptyset,v,\emptyset)$, $\widetilde{T}=\widebar{T^*}$, and jump to Step 3.
\item[Step 2] 
Otherwise $v_2$ must be a descendent of $v$. Now starting with $v$, we find the maximal chain of consecutive $\ominus$-nodes connected by left edges: $w_1=v, w_2,\ldots,w_m$, such that $w_m$ is the root of $T$ or $w_m$ has the right parent which is an $\oplus$-node. (Note that $w_m$ cannot have an $\oplus$-node as its left parent, since this will contradict with the fact that $v$ is the $(l+1)$-th initial $\ominus$-node). Let $T^*$ (resp.~$T_1$) be the tree having $w_m$ and $v$ (resp.~$w_2$) as its root and iroot, respectively. In the special case of $m=1$, simply take $T_1=\emptyset$. Furthermore, we let $\widetilde{T}$ be the {\bf conjugate} of the following tree
$$(T^*\backslash T_1)/T_1(\emptyset,v_2).$$

% Find the lowest $\oplus$-node, denoted $v_1$ (possibly coincidence with $v_2$), in the same {\em left chain} as $v$, where a left chain is a chain linked by only left edges. Let $T_2$ be the left subtree of $v_1$ in $T$ and let $T_1$ be the left subtree of $v$ in $T_2$. Let $T_2'$ be the tree obtained  from $T_2$  by cutting $T_1$. Let $T_1'$ be the tree obtained from $T$ by cutting $T_2$ and attaching $T_1$ at $v_1$ as its new left subtree. 
\item[Step 3] 
Let $w$ be the root of $T$. If $w$ coincides with $v$, then take $\psi^{-1}(T)=(T\backslash T^*)/\widetilde{T}(u_1,\emptyset)$, where $u_1$ is the left child of $v$ in $T$. Otherwise, we take
$$\psi^{-1}(T)=(T\backslash T^*)/\widetilde{T}(w,\emptyset).$$

% Let $v$ together with its right subtree (if any) in $T_2'$ be $T_3$. Let $T_3'$ be the tree obtained from $T_3$ by attaching the reminding  part  of $T_2'$  (after cutting  $T_3$) to the node $v_2$ of $T_3$. Let $T_3^*$ be the di-sk tree  obtained from $T_3'$ by changing all $\oplus$-nodes (resp.~$\ominus$-nodes) into $\ominus$-nodes (resp.~$\oplus$-nodes). 
% Finally, attach $T_1'$ as left subtree  at $v$ of $T_3^*$ to form the di-sk tree $\psi^{-1}(T)$.
\end{itemize}
It is routine to check that $\psi$ and $\psi^{-1}$ are inverse to each other and so $\psi$ is indeed  a bijection. 

By Lemmas~\ref{lem:sep1} and~\ref{lem:sep2}, the pair $(\comp,\idr)$ of Comtet statistics over $\S_n(2413,4213)$ is equidistributed with the pair $(\top,\iom)$ over $\DT_n$. Since $\psi:\DT_n^{(k,l)}\rightarrow \DT_n^{(k-1,l+1)}$ is a bijection, the pair $(\top,\iom)$ is symmetry over $\DT_n$ and so does the pair  $(\comp,\idr)$ over $\S_n(2413,4213)$. 
\end{proof}

%%%%%%%%%%%%%%%%%
\section{Concluding remarks}%%
%%%%%%%%%%%%%%%%%
\label{sec:fin}

The main achievement of this paper is the construction of a combinatorial bijection on di-sk trees that  proves the equidistribution of  two quintuples $(\LMAX,\LMIN,\DESB,\iar,\comp)$ and $(\LMAX,\LMIN,\DESB,\comp,\iar)$  over separable permutations.
At this point, we would like to pose several open problems.
\begin{?}
Our proof of the symmetry of  $(\rpop,\top)$ in Theorem~\ref{thm:rpop} is purely  algebraic, can one find a direct bijective proof (probably in the same spirit as $\psi$ constructed in Theorem~\ref{thm:bij2})? Could the three symmetries in Conjecture~\ref{conj:sym} be proved bijectively? 
\end{?}

The three Comtet statistics $\riop$, $\iop$ and $\top$ in Theorem~\ref{thm:transver} have  interpretations in terms of natural statistics over separable permutations under the bijection $\eta$. This makes us wonder whether there are natural interpretations of $\pop$ and $\rpop$ in terms of separable permutations. 
\begin{?}
Define explicitly two statistics, say $\st$ and $\st'$, for every separable permutation $\pi$, such that
$$\st(\pi)-1 = \pop(\eta(\pi)),\quad \text{and} \quad \st'(\pi)-1 = \rpop(\eta(\pi)).$$
In view of \eqref{eq:iop} and \eqref{five equal}, such two statistics are (new?) Comtet statistics over separable permutations. Similar question can be asked for the statistics $\pom$ and $\rpom$.
\end{?}
\begin{?}
Sitting at the heart of our proof of Theorem~\ref{thm:sep:sym} is di-sk tree, we need the mapping $\eta$ to transform the results back and forth between separable permutations and di-sk trees. This is reminiscent of Rubey's proof \cite{rub} of the equidistribution of $(\LMAX,\iar)$ and $(\LMAX,\comp)$ on $321$-avoiding permutations using Dyck paths, where Krattenthaler's bijection plays the role of $\eta$. So one may ask, is there a way to bypass the use of di-sk tree and prove Theorem~\ref{thm:sep:sym} directly on permutations? This has been done in our previous work \cite{flw} for the case of $321$-avoiding permutations.
\end{?}

Many classical permutation statistics, such as {\em Eulerian} statistics, {\em Mahonian} statistics or {\em Stirling} statistics, have been extensively investigated in the literature (see the excellent book exposition~\cite{kit} of Kitaev) not only on permutations avoiding ordinary  patterns, but also on permutations avoiding consecutive patterns or the more general vincular patterns. It would be interesting to explore systematically the distributions of the two Comtet statistics, $\iar$ and $\comp$, on permutations avoiding vincular patterns.

%%%%%%%%%%%%%%%%%%%%
\section*{Acknowledgement}%%%%%%
%%%%%%%%%%%%%%%%%%%
We thank Martin Rubey for sharing with us his observation that  the set-valued statistic $\LMIN$ could be added to Theorem~1.2.
The second author was supported
by the National Science Foundation of China grants 11871247 and the project of Qilu Young Scholars of Shandong University.


\begin{thebibliography}{99}

\bibitem{adin} R. Adin, E. Bagno, Y. Roichman, Block decomposition of permutations and Schur-positivity, J. Algebraic Combin. {\bf47} (2018), 603--622.

\bibitem{bon} M.~B\'ona, {\em Combinatorics of Permutations}, Chapman \& Hall/CRC, Boca Raton, 2004.

\bibitem{CK} A.~Claesson and S.~Kitaev, Classification of bijections between $321$-and $132$-avoiding permutations, S\'em. Lothar. Combin., {\bf 60} (2008), B60d, 30 pp.



\bibitem{com} L. Comtet, {\em Advanced Combinatorics: The Art of Finite and Infinite Expansions}, D. Reidel Publishing Co., Dordrecht, 1974.

\bibitem{cor} S. Corteel, M. Martinez, C.D. Savage and M. Weselcouch, Patterns in inversion sequences I,  Discrete Math. Theor. Comput. Sci., {\bf18} (2016),  $\#2$.



\bibitem{EP} S. Elizalde and I. Pak, Bijections for refined restricted permutations, J. Combin. Theory Ser. A {\bf105} (2004), 207--219.

\bibitem{FS} D. Foata, M.-P. Sch\"utzenberger, Th\'eorie g\'eom\'etrique des polyn\^omes eul\'eriens, Lecture Notes in Math., vol. 138, Springer-Verlag, Berlin, 1970.



\bibitem{flz} S. Fu, Z. Lin and J. Zeng, Two new unimodal descent polynomials, Discrete Math., {\bf341} (2018),  2616--2626.

\bibitem{flw} S. Fu, Z. Lin and Y. Wang, Refined Wilf-equivalences by Comtet statistics, \href{https://arxiv.org/abs/2009.04269}{arXiv:2009.04269}.

\bibitem{fw} S.~Fu and Y.~Wang, Bijective proofs of recurrences involving two Schr\"oder triangles, European J. Combin. {\bf86} (2020), 103077.

\bibitem{kit} S.~Kitaev, {\em Patterns in permutations and words}, Springer Science \& Business Media, 2011. 

%\bibitem{knu} D.~E.~Knuth, {\em The art of computer programming. Volume 1, Fundamental Algorithms.} Addison-Wesley, Reading, second edition, 1975. Addison-Wesley Series in Computer Science and Information Processing.

\bibitem{kra}C. Krattenthaler, Permutations with restricted patterns and Dyck paths, 
Special issue in honor of Dominique Foata's 65th birthday, Adv. in Appl. Math., {\bf27} (2001), 510--530.

\bibitem{LK} Z.~Lin, D.~Kim, A sextuple equidistribution arising in pattern avoidance, J. Combin. Theory Ser. A {\bf 155} (2018), 267--286.

\bibitem{oeis} OEIS Foundation Inc., The On-Line Encyclopedia of Integer Sequences,  \href{http://oeis.org}{http://oeis.org}, 2020.

\bibitem{pet} T.K.~Petersen, Eulerian numbers. With a foreword by Richard Stanley. Birkh\"auser Advanced Texts: Basler Lehrb\"ucher. Birkh\"auser/Springer, New York, 2015.

\bibitem{rub}M. Rubey, An involution on Dyck paths that preserves the rise composition and interchanges the number of returns and the position of the first double fall, S\'em. Lothar. Combin., Art. B77f, 4 pp.

\bibitem{shapi} L. Shapiro and A.B. Stephens, Bootstrap percolation, the Schr\"oder numbers, and the $N$-kings problem, SIAM J. Discrete Math., {\bf4} (1991), 275--2802.

\bibitem{shapi2} L. Shapiro and R. Sulanke, Bijections for the Schr\"oder numbers, Math. Mag., {\bf73} (2000), 369--376. 

\bibitem{SS} R.~Simion, F.~Schmidt, Restricted permutations, European J. Combin. {\bf6} (1985), 383--406.


\bibitem{sk} Z.~Stankova, Forbidden subsequences, Discrete Math. {\bf132} (1994), 291--316.


\bibitem{stan} R.P. Stanley, {\em Enumerative Combinatorics, vol. 2}, Cambridge Stud. Adv. Math., vol. 62, Cambridge University Press, Cambridge, 1999.


%\bibitem{sta2} Z.~Stankova, Classification of forbidden subsequences of length $4$, European J. Combin. {\bf17} (1996): 501--517.
%
%\bibitem{SW} Z.~Stankova, J.~West, A new class of Wilf-equivalent permutations, J. Algebraic Combin. {\bf15} (2002): 271--290.

%\bibitem{wan} D.~Wang, The Eulerian distribution on involutions is indeed $\gamma$-positive, J. Combin. Theory Ser. A {\bf165} (2019), 139--151.

\bibitem{wes} J.~West, Generating trees and the Catalan and Schr\"oder numbers, Discrete Math. {\bf146} (1995), 247--262.

\end{thebibliography}
\end{document}